\documentclass[11pt]{article}
\usepackage{amsmath, amssymb, theorem, latexsym, epsfig}
\numberwithin{equation}{section}

\theoremstyle{plain}
\theorembodyfont{\itshape}
\newtheorem{theorem}{Theorem}[section]
\newtheorem{proposition}[theorem]{Proposition}
\newtheorem{lemma}[theorem]{Lemma}
\newtheorem{corollary}[theorem]{Corollary}
\theorembodyfont{\rmfamily}
\newtheorem{definition}[theorem]{Definition}

\newtheorem{example}[theorem]{Example}
\newtheorem{remark}[theorem]{Remark}
\newtheorem{convention}[theorem]{Convention}

\newenvironment{proof}{{\noindent \textbf{Proof}\,\,}}{\hspace*{\fill}$\Box$\medskip}

\def\dist{\operatorname{dist}}

\def\wt#1{\widetilde#1}

\def\rr{\mathbb R}
\def\var{\varepsilon}

\def\wh#1{\widehat#1}

\def\nn{\mathbb N}
 \def\mct{\mathcal T}
 \def\mch{\mathcal H}

 \def\zz{\mathbb Z}
 \def\rp{\mathbb{RP}}

\def\La{\Lambda}

\def\mca{\mathcal A}
  \def\diag{\operatorname{diag}}
 \def\la{\lambda}

\def\mcp{\mathcal P}

\def\mcn{\mathcal N}

\def\Xi{\mathcal Z}
\def\mcl{\mathcal L}

\def\mcf{\mathcal F}

\def\gl{\operatorname{GL}}
\def\flat{\operatorname{flat}}

\def\ut{\mathcal{UT}}
\title{On infinitely many foliations by caustics in  strictly convex open  billiards} 
\author{Alexey Glutsyuk\thanks{
CNRS, France (UMR 5669 (UMPA, ENS de Lyon), UMI 2615 (ISC J.-V.Poncelet)). E-mail: 
aglutsyu@ens-lyon.fr} \thanks{HSE University, Moscow, Russian Federation} 
\thanks{Kharkevich Institute for Information Transmission Problems (IITP RAS), Moscow, Russia}
 \thanks{Partly supported by Laboratory of Dynamical Systems and Applications, HSE University, of the Ministry of science and higher education of RF grant ag. No 075-15-2019-1931}\thanks{Partly supported by RFBR according to the research project 20-01-00420}}
\begin{document}
\maketitle
\begin{abstract} Reflection in strictly convex bounded planar  billiard acts 
on the space of oriented lines and preserves a standard area form. 
A {\it caustic}  is a curve $C$ whose tangent lines 
are reflected by the billiard to lines tangent to $C$. The famous Birkhoff conjecture 
states that the only strictly convex billiards with a foliation by closed caustics near the boundary are ellipses. By Lazutkin's theorem,  there always exists a Cantor 
family of closed caustics approaching the boundary. In the present paper we 
deal with  an open billiard, whose boundary is a  strictly convex embedded 
(non-closed) curve $\gamma$. We prove that there exists a domain $U$ adjacent to $\gamma$ 
from the convex side and a $C^\infty$-smooth foliation of  $U\cup\gamma$ whose leaves are $\gamma$ 
and (non-closed) caustics of the billiard.  
This generalizes a previous result by R.Melrose on existence of 
a germ of  foliation as above. 
We show that there exist a continuum of above foliations by caustics whose 
germs at each point in $\gamma$ are pairwise different. 
We prove  a more general version of this 
statement 
for $\gamma$ being an (immersed) arc. 
It also applies  to a billiard bounded by a closed strictly convex 
curve $\gamma$ and yields infinitely many 
"immersed" foliations by immersed caustics.
For the proof of the above results, we state and prove their analogue 
for a special class of area-preserving maps generalizing 
billiard reflections: the so-called $C^{\infty}$-lifted strongly billiard-like maps. We also prove a series of results on 
conjugacy of billiard maps  near the boundary for open curves of the above type.
\end{abstract}
\tableofcontents

\section{Introduction and main results}

The billiard reflection  from a strictly convex smooth planar curve $\gamma\subset\rr^2$ 
(parametrized by either a circle, or an interval) is a map $\mct$ acting on the  
subset in the space of oriented lines that consists of 
those lines that are either tangent to $\gamma$, or intersect $\gamma$ transversally at two points. (In general, 
the latter subset is not $\mct$-invariant. In the case, 
when $\gamma$ is a closed curve, the latter subset is $\mct$-invariant and called the {\it phase cylinder.}) 
Namely, 
if a line is tangent to $\gamma$, then it is a fixed point of the reflection map. If a line $L$ intersects $\gamma$ 
transversally at two points, take its last intersection point $B$ with $\gamma$ (in the sense of orientation of the line $L$) 
and reflect $L$ from $T_B\gamma$  according to the usual reflection law: the angle of incidence is equal to 
the angle of reflection. By definition, the image $\mct(L)$ is the reflected line  oriented at $B$ inside the convex domain adjacent to $\gamma$. 
The reflection map $\mct$  is  
called the {\it billiard ball map.} See Fig. 1.

The space of oriented lines in Euclidean plane $\rr^2_{x,y}$ is homeomorphic to cylinder,  and  it carries the standard symplectic form 
\begin{equation}\omega=d\phi\wedge dp,\label{defom}\end{equation} 
where $\phi=\phi(L)$ is the azimuth of the line $L$ (its angle  with the $x$-axis) and 
$p=p(L)$ is its signed distance  to the origin $O$ defined as follows. For  each oriented line $L$ 
that does not pass through $O$ consider the circle centered at $O$ and tangent to $L$. We say that $L$ is 
{\it clockwise (counterclockwise)}, if it orients the  latter circle clockwise (counterclockwise). By definition,

-  $p(L)=0$, if and only if $L$ passes through the origin $O$; 

-  $p=\dist(L,O)$, if $L$ is clockwise;  otherwise  $p=-\dist(L,O)$. 

It is well-known that 

- the symplectic form $\omega$ is invariant under affine orientation-preserving 
isometries;

- {\it the billiard reflections from all planar curves preserve the symplectic form $\omega$.}

\begin{definition} A curve $C$ is a {\it caustic} for the billiard on the curve $\gamma$, if each line tangent to $C$ 
is reflected from $\gamma$ to a line tangent to $C$. Or equivalently, if the curve of (appropriately oriented) 
tangent lines to $C$ is an invariant curve for the billiard ball map. See Fig. 1. 
\end{definition}
\begin{figure}[ht]
  \begin{center}
   \epsfig{file=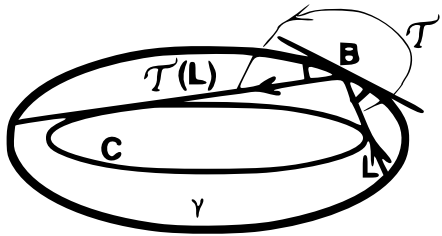, width=15em}
    \caption{The billiard ball map and a caustic.}
    \label{fig:01}
  \end{center}
\end{figure}
The famous Birkhoff Conjecture deals with a planar billiard bounded by a strictly convex closed curve $\gamma$. Recall that 
such  a billiard is called {\it Birkhoff integrable,} if there exists a topological annulus 
adjacent to $\gamma$ from the convex side foliated by closed caustics, 
and $\gamma$ is a leaf of this foliation. See Figure 2.  
It is well-known that the billiard in an ellipse is integrable, since it has a family of closed  caustics: confocal 
ellipses. 
The {\bf Birkhoff Conjecture} states the converse: {\it the only integrable planar billiards are ellipses.} 
\begin{figure}[ht]
  \begin{center}
   \epsfig{file=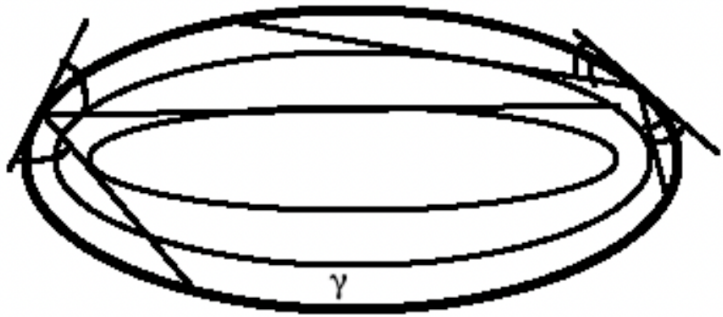, width=10cm}
    \caption{A Birkhoff integrable billiard.}
    \label{fig:02}
  \end{center}
\end{figure}

 \begin{remark} The condition of the Birkhoff Conjecture stating that the caustics in question form a {\it foliation} is important: 
 the famous result by Vladimir Lazutkin (1973) states that {\it each strictly convex bounded planar billiard with boundary  smooth enough 
 has a Cantor family of closed caustics.} But Lazutkin's caustic family  does not extend to a foliation in general. 
 \end{remark}
 
The main result of the paper presented in Subsection 1.1 shows that the other condition of the Birkhoff Conjecture stating  that the caustics in question 
are {\it closed} is also important: the Birkhoff Conjecture is false without closeness condition. Namely we show that 
any open strictly convex $C^\infty$-smooth planar  curve $\gamma$ has an adjacent domain $U$ (from the convex side) admitting a 
 foliation by caustics of $\gamma$ that extends to a $C^\infty$-smooth foliation  of the domain with boundary 
$U\cup\gamma$ with $\gamma$ being a leaf.  Moreover, we show that $U$ can be chosen so that 
there exist infinitely many (continuum of)  such foliations, 
and any two distinct foliations have pairwise distinct germs at every point in $\gamma$. We  prove analogous statement for a non-injectively immersed curve $\gamma$ and  "immersed foliations" by immersed caustics. We state and prove an analogue of this statement in the special case, 
when $\gamma$ is a closed curve. 

\begin{remark} \label{rkab} Consider the map $\mct$ of  billiard reflection from a strictly convex planar oriented 
$C^{\infty}$-smooth curve $\gamma$ that is a one-dimensional submanifold in $\rr^2$ parametrized by interval. 
 Let $\wh\gamma$ denote the family of its orienting tangent lines. Then the points of the curve $\wh\gamma$ are fixed by $\mct$. The map $\mct$ 
 is a well-defined area-preserving map 
 on an open subset adjacent to $\wh\gamma$ in the space of oriented lines. The latter subset  consists of those lines that intersect $\gamma$ 
 transversally and are directed to the concave side from $\gamma$ at some intersection point. Each caustic close to $\gamma$ 
 corresponds to a  $\mct$-invariant curve (the family of its tangent lines chosen with appropriate orientation) and vice versa. 
 Thus, a foliation by caustics induces a foliation by $\mct$-invariant curves. In Subsection 2.7 we prove the converse: 
 each $C^{\infty}$-smooth foliation by $\mct$-invariant curves on a domain adjacent to $\wh\gamma$ from 
 appropriate side (with $\wh\gamma$ 
 being a leaf) induces a $C^{\infty}$-smooth foliation by caustics (with $\gamma$ being a leaf).
 \end{remark}

We show that {\it the billiard map has infinite-dimensional family of $C^{\infty}$-smooth foliations by invariant curves 
(including $\wh\gamma$) 
in appropriate domain adjacent to $\wh\gamma$ with pairwise distinct germs at each point of the 
curve $\wh\gamma$.} This together with Remark \ref{rkab} implies  existence of infinite-dimensional 
family of foliations by caustics.

In Subsection 1.3 we state the generalization of the above result on foliations by invariant curves  to a special class of area-preserving maps: 
the so-called $C^{\infty}$-lifted strongly billiard-like maps, for which we prove existence of infinite-dimensional 
family of  $C^{\infty}$-smooth foliations by invariant curves with pairwise distinct germs at each point 
of the boundary segment. In Subsection 1.4 we describe one-to-one correspondence 
between germs of the latter foliations 
and germs at $S^1\times\{0\}$ of $C^{\infty}$-smooth $h$-flat functions $\psi(t,h)$ on the cylinder $S^1\times\rr_{\geq0}$ 
such that $\psi(0,h)\equiv0$. 
This yields a one-to-one correspondence between foliations by caustics and the above 
germs of flat functions on cylinder. 
Theorem \ref{unget} stated in Subsection 1.4 asserts that all the foliations by caustics (invariant curves) corresponding 
to a given billiard (map) have coinciding jets of any order at each point of the boundary curve. 

 The  results of the paper mentioned below are motivated  
by the following open question attributed to Victor Guillemin:

{\it Let two billiard maps corresponding to two strictly convex closed Jordan curves be conjugated by a homeomorphism. What can be said about the curves? Are they similar (i.e., of the same shape)?}

Theorem \ref{addthm3} presented in Subsection 1.3 states that each 
$C^{\infty}$-lifted strongly billiard-like map is $C^\infty$-smoothly symplectically 
conjugated near the boundary (and up to the boundary) to the normal form $(t,z)\mapsto(t+\sqrt z,z)$ 
restricted to $U\cup J$, where $J\subset\rr\times\{0\}$ is an interval of the horizontal axis and 
$U\subset\rr\times\rr_+$ is a domain adjacent to $J$.  
In particular, this holds for the billiard map corresponding to each $C^\infty$-smooth strictly convex 
(immersed) curve. 
As an application, we obtain a series of results on (symplectic) conjugacy of  billiard maps near the boundary 
for billiards with reflections from $C^\infty$-smooth strictly convex curves parametrized by intervals. 
These conjugacy results are stated in Subsection 1.5 and proved in Subsection 2.10. One of them (Theorem \ref{thconj1}) 
states that for any two strictly convex open billiards, each of them being bounded by an infinite curve with asymptotic tangent line at infinity 
in each direction, the corresponding billiard maps are $C^\infty$-smoothly conjugated near the boundary. 

 The results of the paper are proved in Section 2. The plan of proofs  is presented in Subsection 1.6. 
The corresponding background material on symplectic 
properties of billiard ball map is recalled in Subsection 1.2. A brief historical survey is presented in Subsection 1.7.

\subsection{Main result: an open convex arc has infinitely many foliations by caustics}

 Consider an open planar billiard: a convex planar domain bounded by a strictly convex 
$C^\infty$-smooth one-dimensional submanifold $\gamma$ that is a curve  parametrized by interval; it goes 
 to infinity in both directions. 
Let $U$ be a domain adjacent to $\gamma$ from the convex side. Consider a foliation $\mcf$ 
of the domain $U$ by strictly convex smooth curves, 
with $\gamma$ being a leaf. We consider 
that it is a foliation by (connected components of) level curves of a   continuous 
function $h$ on $U\cup\gamma$  such that $h|_\gamma=0$, $h|_U>0$ and $h$ strictly increases as a function of 
the transversal parameter. We also consider that for every  $x\in\gamma$ and every leaf $\mcl$ 
of the foliation $\mcf$ there are at most two tangent lines to $\mcl$ through $x$. One can achieve this by shrinking 
the foliated domain $U$, since for every $x\in\gamma$ the line $T_x\gamma$ is the only line through $x$ 
tangent to $\gamma$. Indeed, if there were another line through $x$ tangent to $\gamma$ at a 
point $y\neq x$, then the total increment of azimuth of the orienting tangent vector to $\gamma$ along the 
arc $xy$ would be greater than $\pi$. But the latter azimuth is monotonous, and its total increment 
along the curve $\gamma$ is no greater than $\pi$, since $\gamma$ is convex  and goes to infinity in both 
directions. The contradiction thus obtained proves uniqueness of tangent line through $x$. 

\begin{remark} In the above conditions 
for every compact subarc $\gamma'\subset\gamma$ and every leaf $\mcl$  of the foliation $\mcf$ 
close enough to $\gamma$ for every $x\in\gamma'$ there exist exactly two tangent lines to $\mcl$ through $x$.  
This follows from convexity. 
\end{remark}

\begin{definition} 
 We say that $\mcf$ is  
a {\it foliation by caustics} of the billiard played on $\gamma$, if its leaves are  caustics, see Fig. 3,  in the following sense. Let $x\in\gamma$, and let $\mcl$ be a leaf of the foliation $\mcf$. If  there exist two tangent lines 
to $\mcl$ through $x$, then they are 
symmetric with respect to the tangent line $T_x\gamma$. 
 \end{definition}

\begin{remark}\label{extralines} The above definition 
 also makes sense in the case, when $\gamma$ is just a strictly convex arc that needs not go to infinity. 
 A priori, in this case for some $x\in\gamma$ there may be more than two tangent lines through $x$ to a leaf 
 of the foliation, even for leaves arbitrarily close to $\gamma$. This holds, e.g., if there is a line through $x$ 
 tangent to $\gamma$ at a point distinct from $x$. This may take place only in the case, when
  the azimuth increment along $\gamma$ of the orienting tangent vector to $\gamma$ is bigger than $\pi$. 
  In this case we modify the above definition as follows. Let $\mathcal H$ denote the space of triples 
 $(x,y,z)$, where $x\in \gamma$ and $y$, $z$  lie in the same leaf $\mcl$ of the foliation  $\mcf$, $y\neq z$, such that 
the lines $xy$ and $xz$  are 
 tangent to  $\mcl$ at the points $y$ and $z$ respectively. Set 
 $$\overline\mch:=\mch\cup\Delta, \ \ \ \Delta:=\{(x,x,x) \ | \ x\in \gamma\}.$$
  Let $\mch_0$ denote the  
path-connected component of the space $\overline\mch$ that contains $\Delta$. 
 We require that for every $(x,y,z)\in\mch_0\setminus\Delta$  
 the lines $xy$ and $xz$ be symmetric 
 with respect to the line $T_x\gamma$.
\end{remark} 
\begin{figure}[ht]
  \begin{center}
   \epsfig{file=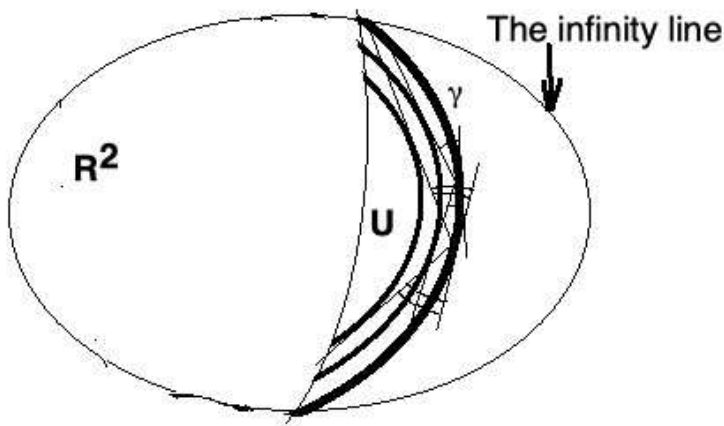, width=15cm}
    \caption{An open strictly convex planar billiard and its caustics. Here the ambient plane $\rr^2$ is presented 
    together with its boundary: the infinity line.}
    \label{fig:01}
  \end{center}
\end{figure}
\begin{definition} \label{dample} Let $\gamma\subset\rr^2$ be a smooth curve parametrized by an interval. 
Let $U\subset\rr^2$ be a domain adjacent to $\gamma$. A collection of $C^\infty$-smooth foliations on $U\cup\gamma$ with 
$\gamma$ being a leaf 
is  said to be an {\it infinite-dimensional family of foliations with distinct boundary germs,} if their 
germs at each point in $\gamma$ are pairwise distinct, and if their collection contains 
a $C^\infty$-smooth $N$-parametric family of foliations for every $N\in\nn$. 
\end{definition} 

\begin{theorem} \label{thm1} 1) Consider an open billiard bounded 
by a strictly convex $C^{\infty}$-smooth curve $\gamma\subset\rr^2$:  
a one-dimensional submanifold parametrized by interval. 
There exists a simply connected domain $U$ adjacent to $\gamma$ from the convex side that admits a 
  foliation by caustics of the billiard that extends to a $C^\infty$-smooth foliation on 
$U\cup\gamma$, with $\gamma$ being a leaf. Moreover, $U$ can be chosen to admit an 
infinite-dimensional family of foliations as above with distinct boundary germs.  See Fig. 3. 

2) The above statements remain valid in the case, when $\gamma$ is just an arc: a strictly convex 
curve parametrized by an interval such that each its point has a neighborhood $V$ whose intersection with 
$\gamma$ is a submanifold in $V$.  
\end{theorem}

\begin{remark} \label{remmelr} It follows from R.Melrose's result \cite[p.184, proposition (7.14)]{melrose1} that each point of the curve $\gamma$ has an arc neighborhood $\alpha\subset\gamma$ 
for which there exists a domain $U$  adjacent to $\alpha$ from the convex side such that $U\cup\alpha$ is 
 $C^{\infty}$-smoothly foliated by caustics of the billiard played on $\gamma$. 
The new result given by Theorem \ref{thm1} is the statement 
that the latter holds  for the whole curve $\gamma$ and there exist infinitely many 
foliations by caustics with distinct boundary germs.
\end{remark}
Below we extend Theorem \ref{thm1} to the case of immersed (or closed) curve $\gamma$. 
\begin{definition} 
Let $\gamma\subset\rr^2$ be a strictly convex   $C^{\infty}$-smooth curve 
that is the image of an interval $(0,1)$ with coordinate $x$ under an {\bf immersion} $\psi:(0,1)\to\gamma$.  
Let $V\subset(0,1)\times\rr_+\subset\rr^2$ be a domain adjacent to the interval $J:=(0,1)\times\{0\}$. Fix a 
 $C^{\infty}$-smooth immersion $\Psi: V\cup J\to\rr^2$ extending $\psi$ as a map  
$J\to\gamma$, sending  $V$  to the convex side 
from $\gamma$. Let $U\subset V$ be a domain adjacent to $J$ and equipped with a foliation  $\mcf$ 
by smooth curves parametrized by intervals, with $J$ being a leaf. We consider that $\mcf$ is a foliation by
 level curves of a continuous function $h:U\to\rr$, $h|_J=0$, $h|_U>0$, such that $h$ strictly increases as a 
 function of the transversal parameter. We say that $\mcf$ is a {\it foliation by lifted caustics} of the billiard 
 played on $\gamma$, if $\Psi$ sends each its leaf $\mcf_t=\{ h=t\}$ to a caustic of the billiard, see Fig. 4. 
In more detail,   let $\mathcal H$ denote the space of triples 
 $(x,y,z)$, where $x\in J$ and $y$, $z$  lie in the same leaf $\mcl$ of the foliation  $\mcf$, $y\neq z$, such that 
the lines $\Psi(x)\Psi(y)$ and $\Psi(x)\Psi(z)$  are 
 tangent to  the curve $\Psi(\mcl)$ at the points $\Psi(y)$ and $\Psi(z)$ respectively. Set 
 $$\overline\mch:=\mch\cup\Delta, \ \ \ \Delta:=\{(x,x,x) \ | \ x\in J\}.$$
  Let $\mch_0$ denote the  
path-connected component of the space $\overline\mch$ that contains $\Delta$. 
 We require that for every $(x,y,z)\in\mch_0\setminus\Delta$  
 the lines $\Psi(x)\Psi(y)$ and $\Psi(x)\Psi(z)$ be symmetric 
 with respect to the line tangent to $\gamma$ at $\Psi(x)$. 
 \end{definition}
 
 \begin{figure}[ht]
  \begin{center}
   \epsfig{file=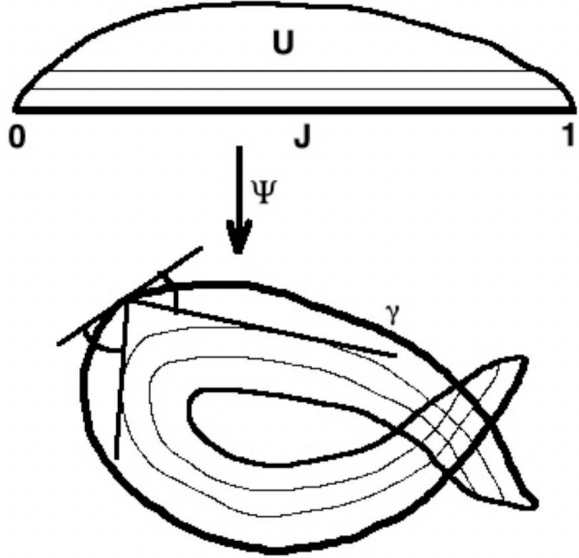, width=8cm}
    \caption{An immersed foliation by immersed caustics.}
    \label{fig:01}
  \end{center}
\end{figure}

 \begin{theorem} \label{thm2} Let $\gamma$, $\psi$, $\Psi$, $J$, $V$ be as above. 
 There exists a domain $U\subset V$ adjacent to $J$ on which there exists
  a  foliation by lifted caustics that extends to a 
$C^\infty$-smooth foliation on $U\cup J$,  with $J$ being a leaf. 
The above $U$ can be chosen so that it admits an infinite-dimensional family of foliations as above 
with distinct boundary germs. See Fig. 4.
\end{theorem}

\begin{theorem} \label{thm2closed} Let 
$\gamma$ be a strictly convex  closed curve bijectively parametrized by    
circle. Fix  a topological annulus $\mca$ adjacent to $\gamma$ from the convex side. 
 Let  $\pi:\wt\mca=\rr\times[0,\var)\to\mca$ be its universal covering, set  
  $J:=\rr\times\{0\}$; $\pi:J\to\gamma$ is the universal covering over $\gamma$. 
 There exists a domain $U\subset\wt\mca\setminus J$ adjacent to $J$ that admits 
a  foliation  by lifted caustics  of the billiard in $\gamma$ that extends to a $C^{\infty}$-smooth foliation on 
$U\cup J$, with  $J$ being a leaf. Moreover, one can choose $U$ so that there  exist an 
infinite-dimensional family of foliations as above with distinct boundary germs.
\end{theorem} 
\begin{remark} In general, in Theorem \ref{thm2closed} the projected leaves are caustics that need not  be closed, may intersect each other and 
may have self-intersections. 
Each individual caustic may have a finite length. However the latter finite length tends to infinity, as the caustic 
in question tends to $\gamma$.
\end{remark}

A generalization of  Theorems \ref{thm1}, \ref{thm2} for the so-called $C^{\infty}$-lifted 
strongly billiard-like maps will be stated in Subsection 1.3. 

\subsection{Background material: symplectic properties of billiard ball map}

Let $\gamma$ be a $C^{\infty}$-smooth strictly convex oriented curve in $\rr^2$ parametrized injectively either by an interval, or by circle. Let $s$ be its natural length parameter 
respecting its orientation. We identify a point in $\gamma$ with the corresponding value of the natural parameter $s$. 

Let $\Gamma:=T_{=1}\rr^2|_\gamma\subset T\rr^2_\gamma$ 
denote the restriction to $\gamma$ of the unit tangent 
bundle of the ambient plane $\rr^2$: 
$$\Gamma=\{(q,u) \ | \ q\in\gamma, \ u\in T_q\rr^2, \ ||u||=1\}.$$
 It is a two-dimensional surface parametrized diffeomorphically by 
$(s,\phi)\in\gamma\times S^1$; here $\phi=\phi(u)$ is the angle of 
a given unit tangent vector $u\in T_s\rr^2$   with the orienting unit tangent vector $\dot\gamma(s)$ to $\gamma$. The curve 
$$\wt\gamma:=\{\phi=0\}=\{ (s,\dot\gamma(s)) \ | \ s\in\gamma\}$$
is the graph of the above vector field $\dot\gamma$. 
For every $(q,u)\in\Gamma$ set 
$$L(q,u):=\text{  the oriented line through } q \text{ 
directed by the vector } u.$$
 We treat the two  following cases separately. 

{\bf Case 1):}  the curve $\gamma$  either is parametrized by 
an interval and goes to infinity in both directions, or is parametrized by circle. 
That is, it bounds a strictly convex infinite (respectively, bounded) 
planar domain. Let $\Gamma^0\subset\Gamma$ 
denote  the neighborhood of the curve $\wt\gamma$ 
that consists of those   $(q,u)\in\Gamma$ that satisfy the  following conditions: 

a) the line $L(q,u)$ either intersects $\gamma$ at two points $q$ and $q'$, or  is the orienting 
tangent line to $\gamma$ at $q$: $u=\dot\gamma(s)$; 
in the latter case we set $q':=q$;

b)  the angle between the oriented line $L(q,u)$ and any of the orienting tangent vectors to 
$\gamma$ at $q$ or $q'$ is acute\footnote{In the case under consideration condition b) implies that 
the line $L(q,u)$ has acute angle with the orienting tangent vector $\dot\gamma$ at each point of the arc $qq'$ (for appropriately 
chosen arc $qq'$ in the case, when $\gamma$ is a  closed curve).}

Let $u'$ denote the directing unit vector of the line $L(q,u)$ at $q'$. Consider the two following involutions acting on $\Gamma^0$ 
and $\Gamma$ respectively:
$$\beta:\Gamma^0\to\Gamma^0, \ \beta(q,u)=(q',u'); \ \ \beta^2=Id;$$
$$I:\Gamma\to\Gamma \text{ is the reflection from } T_q\gamma: \  I(q,u)=(q,u^*),$$
where $u^*$ is the vector symmetric to $u$ with respect to the tangent line $T_q\gamma$.
Let $\Gamma^0_+\subset\Gamma^0$ denote the open subset of those pairs $(q,u)$ 
in which the vector $u$ is directed to the convex side from the curve $\gamma$.

\begin{remark} The domain $\Gamma^0$ is  $\beta$-invariant.  It is 
a topological disk (cylinder), if $\gamma$ is parametrized by an interval (circle). 
The domain $\Gamma^0_+$ is a topological disk (cylinder) adjacent to 
$\wt\gamma$.
\end{remark}

Let $\Pi_\gamma$ denote the open subset of the space of oriented lines in $\rr^2$ 
consisting of the  lines $L(q,u)$ with $(q,u)\in\Gamma^0_+$.   The 
mapping $\La:(q,u)\mapsto L(q,u)$ is a  diffeomorphism 
$$\La:\Gamma^0_+\to\Pi_\gamma$$ 
It extends  to the set $\Gamma^0_+\cup\wt\gamma$ as a  homeomorphism 
sending each point
  $(s,\dot\gamma(s))\in\wt\gamma$ to the tangent line $T_s\gamma$ directed by $\dot\gamma(s)$. 
  \begin{remark} \label{1.9} Let $\mct$ denote the billiard ball map given by reflection from 
  the curve $\gamma$ acting on oriented lines. 
  It is well-known that the billiard ball map $\mct$ restricted to $\Pi_\gamma$ 
    is conjugated  by $\La$ to the product of two involutions
  $$\wt\delta_+:=I\circ\beta=\La^{-1}\circ\mct\circ \La:\Gamma^0_+\to\Gamma.$$
 If the curve $\gamma$  is $C^{\infty}$-smooth, then both involutions $I$ and 
 $\beta$ are $C^{\infty}$-smooth on $\Gamma$ and $\Gamma^0$ respectively. Their 
 product is well-defined and smooth on a neighborhood of the curve $\wt\gamma$ 
 and fixes the points of the curve $\wt\gamma$. Both involutions preserve 
 the canonical symplectic form $\sin\phi ds\wedge d\phi$ on 
 $\Gamma\setminus\wt\gamma$, which is known to be the $\La$-pullback 
 of the standard symplectic form on the space of oriented lines. See \cite{ar2, ar3, mm, 
 melrose1, melrose2, tab95}; see also \cite[subsection 7.1]{gpor}.
 \end{remark} 

Let us recall another representation of  the billiard ball map $\mct$ 
 in a chart where it preserves the standard symplectic form. To do this, consider 
 the orthogonal projection $\pi_\perp:(T\rr^2)|_\gamma\to T\gamma$ sending 
 each vector $u\in T_q\rr^2$ with $q\in\gamma$ to its orthogonal projection to 
 the tangent line $T_q\gamma$. It projects  the unit tangent bundle $\Gamma$ 
 to the unit ball bundle 
 $$T_{\leq1}\gamma:=\{(q,w) \ | q\in\gamma, \ w\in T_q\gamma, \ ||w||\leq1\}.$$
A tangent vector $w=w\frac{\partial}{\partial s}\in T_q\gamma$ will be identified with its coordinate $w=\pm||w||$ in the basic vector $\frac{\partial}{\partial s}$. Thus, 
$\pi_\perp(s,\phi)=(s,\cos\phi)$. 
Consider the following function and differential form on $T\gamma$:
\begin{equation}y:=1-w; \ \omega:=ds\wedge dy.\label{defy}\end{equation}
The form $\omega$ coincides with  the standard symplectic form on the tangent bundle 
$T\gamma$ of the curve $\gamma$ (considered as a Riemannian manifold 
equipped with the metric $|ds|^2$ coming from the standard Euclidean metric on $\rr^2$).

The  curve 
$\wt\gamma=\{(s,\dot\gamma(s)) \ | \ s\in\gamma\}=
\{ w=1\}=\{ y=0\}\subset T\gamma$
 is a component of the boundary $\partial T_{\leq1}\gamma$.
The projection $\pi_\perp$ sends $\Gamma^0_+$ diffeomorphically to a domain 
 in $T_{\leq1}\gamma$ adjacent to $\wt\gamma$. It extends homeomorphically 
 to $\Gamma^0_+\cup\wt\gamma$ as the identity map $Id:\wt\gamma\to\wt\gamma$. 
 Let $\mu_+: \pi_\perp(\Gamma^0_+\cup\wt\gamma)\to \Gamma^0_+\cup\wt\gamma$ be the  inverse 
to the restriction of the projection $\pi_\perp$ to $\Gamma^0_+\cup\wt\gamma$. Set 
\begin{equation}\delta_+:=\pi_\perp\circ\wt\delta_+\circ\mu_+=\pi_\perp\circ \La^{-1}\circ\mct\circ \La\circ\mu_+.\label{defb}\end{equation}

\begin{theorem} \label{tsym} (\cite[subsection 1.5]{tab95}, \cite{melrose1, melrose2, 
ar2, ar3}; see also \cite[theorem 7.3]{gpor}). 
The mapping  $\delta_+:\pi_\perp(\Gamma^0_+)\to T_{\leq1}\gamma$ 
given by (\ref{defb}), 
is symplectic: it preserves the form $\omega=ds\wedge dy$. 
\end{theorem} 

\begin{proposition} \label{psmi} \cite[proposition 7.5]{gpor}. Let $\kappa(s)$ 
denote the (geodesic) curvature of the curve $\gamma$. 
The involutions $I$, $\beta$ and the mappings 
$\wt\delta_+$, $\delta_+$ admit the following (asymptotic) formulas:
 \begin{equation}I(s,\phi)=(s,-\phi), \ \beta(s,\phi)=(s+2\kappa^{-1}(s)\phi+O(\phi^2), -\phi+O(\phi^2)),\label{beti}\end{equation}
 \begin{equation} \wt\delta_+(s,\phi)=(s+2\kappa^{-1}(s)\phi+O(\phi^2), \phi+O(\phi^2)),\label{bilike}\end{equation}
 \begin{equation}\delta_+(s,y)=(s+2\sqrt 2\kappa^{-1}(s)\sqrt y+O(y), y+O(y^{\frac32})).\label{deltsy}\end{equation} 
 The asymptotics are uniform on compact subsets of points $s\in\gamma$, as 
 $\phi\to0$ (respectively, as $y\to0$). 
\end{proposition}

{\bf Case 2).} Let $\gamma$ be parametrized by an interval, but now it does not 
necessarily go to infinity or 
bound a region in the plane. Moreover, we allow $\gamma$ to 
be an immersed curve that may self-intersect. In this case some lines $L(q,u)$ may 
intersect $\gamma$ in more than two points. 
Now the  definition of the subset $\Gamma^0\subset\Gamma$ should be 
modified to be the subset of those $(q,u)\in\Gamma$ 
for which there exists a $q'\in\gamma\cap L(q,u)$ satisfying the 
condition b) from Case 1) and such that the arc $qq'\subset\gamma$ is 
disjoint from the line $L(q,u)$,  injectively immersed (i.e., without self-intersections) 
and satisfies the statement of Footnote 1: the orienting tangent vector $\dot\gamma$ at each its point has acute 
angle with $L(q,u)$. 
 (Here   $q$ and $q'$ may be not the only points of intersection  $\gamma\cap L(q,u)$.) 

\begin{remark} \label{rkimm} For any given  $(q,u)\in\Gamma^0$ the 
point  $q'$ satisfying the conditions from the above paragraph exists, whenever $u$ is close enough to $\dot\gamma(q)$ (dependently on $q$). Whenever it exists, it is unique. 
All the statements and discussion in the previous Case 1) remain valid in our Case 2). 
Now the mapping $\Lambda$ is 
a local diffeomorphism but not necessarily a global diffeomorphism: an oriented line 
intersecting $\gamma$ at more than 
two points (if any)  may correspond to at least two different tuples $(q,u)\in\Gamma^0_+$. 
\end{remark}

\subsection{Generalization to $C^{\infty}$-lifted strongly billiard-like  maps}

In this subsection and in what follows we study the next class of area-preserving mappings introduced in \cite{gpor} generalizing the billiard maps  (\ref{deltsy}). 
\begin{definition} \label{tdw} (see \cite[definition 7.6]{gpor}). Let $(a,b)$ be a (may be (semi) infinite) interval 
in $\rr$ with coordinate $s$. Let $V\subset\rr\times\rr_{+}$ be a domain adjacent to 
the interval $J:=(a,b)\times\{0\}$. A mapping $F:V\cup J\to\rr\times\rr_{\geq0}\subset\rr^2_{s,y}$ is called {\it billiard-like,} if it satisfies the following conditions:

(i) $F: V\cup J\to F(V\cup J)$ is a homeomorphism  fixing the points in $J$; 

(ii) $F|_V$ is a diffeomorphism preserving the standard area form $ds\wedge dy$;

(iii) $F$ has the asymptotics  of the type 
\begin{equation}F(s,y)=(s+w(s)\sqrt y+O(y),y+O(y^{\frac32})),   \text{ as } y\to0; 
\ w(s)>0,\label{fbm}\end{equation}
uniformly on compact subsets in the $s$-interval $(a,b)$;

(iv) the  variable change 
$$(s,y)\mapsto (s,z), \ z=\sqrt y>0$$ 
conjugates $F$ to a  smooth map $\wt F(s,z)$ (called its {\it lifting}) that is also smooth at points of the boundary interval $J$; thus, $w(s)$ is continuous on $(a,b)$.

 If, in addition to conditions (i)--(iv), the latter mapping $\wt F$ 
is a product of two  involutions $I$ and $\beta$ 
 fixing the points of the line $z=0$, 
$$\wt F=I\circ\beta, \ I(s,z)=(s,-z),$$
\begin{equation}\beta(s,z)=(s+w(s)z+O(z^2), -z+O(z^2)), \ \beta^2=Id,
\label{prinv}\end{equation}
then $F$ will be called {\it a (strongly) billiard-like  map.} 

If $F$ is strongly billiard-like, and the corresponding involution $\beta$ (or equivalently, 
the conjugate map $\wt F$) is $C^{\infty}$-smooth, and also $C^{\infty}$-smooth  at the 
points of the boundary interval 
$J$, then  $F$ is 
called {\it $C^{\infty}$-lifted.}
 The above definitions make sense for $F$ being a germ of map at the interval $J$.
\end{definition}

\begin{example} \label{exdel} 
The mapping $\delta_+$ from (\ref{deltsy})  
is strongly billiard-like in the coordinates 
$(s,y)$ with $w(s)=2\sqrt 2\kappa^{-1}(s)$, see  (\ref{beti}),  (\ref{bilike}) 
and (\ref{deltsy}). If the curve $\gamma$ is $C^{\infty}$-smooth, then $\beta$ 
and hence, $\wt\delta_+=I\circ\beta$ are $C^{\infty}$-smooth, and hence, 
$\delta_+$ is $C^{\infty}$-lifted.
\end{example}

\begin{proposition} \label{classinv} 
The class of (germs at $J$ of) $C^{\infty}$-lifted strongly billiard-like maps is invariant under conjugacy by (germs at $J$ of) $C^{\infty}$-smooth 
symplectomorphisms $G:V\cup J\to G(V\cup J)\subset\rr\times\rr_{\geq0}$ 
sending $J$ onto an interval in $\rr\times\{0\}$. Here $V\subset\rr\times\rr_+$ 
is a domain adjacent to $J$. 
 \end{proposition}
\begin{proof} Let $F$ be a $C^{\infty}$-lifted 
strongly billiard-like map, $\wt F=I\circ\beta$ be its 
lifting. Let $V\subset\rr\times\rr_+$ be a domain adjacent to $J$. Let $F$ be defined on $V\cup J$,  
and let $G:V\cup J\to G(V\cup J)\subset\rr\times\rr_{\geq0}$, 
be a $C^{\infty}$-smooth symplectomorphism as above.
Let us denote $G(s,y)=(\wh s(s,y), \wh y(s,y))$.  
One has $\wh y(s,0)\equiv0$, $\frac{\partial\wh s}{\partial s}(s,0)>0$,  $\frac{\partial\wh y}{\partial y}(s,0)>0$, by definition and  orientation-preserving property (symplecticity). Thus, 
$\wh y(s,y)=yg(s,y)$, where $g(s,y)$ is a positive  $C^{\infty}$-smooth function 
on a neighborhood of the interval $J$ in $(\rr\times\rr_{>0})\cup J$. 
The lifting $\wt G$ of the map $G$ to the variables $(s,z)$, $z=\sqrt y$, acts as follows: 
\begin{equation}\wt G:(s,z)\mapsto(\wh s(s,z^2), \wh z(s,z)); \ \ \wh z=\sqrt{\wh y(s,z^2)}=z\sqrt{g(s,z^2)}.\label{wtg}\end{equation} 
The latter square root is well-defined and $C^{\infty}$-smooth. This implies 
that the map $\wt G$ is a $C^{\infty}$-smooth diffeomorphism 
of domains with arcs of boundaries corresponding to $V\cup J$ and $G(V\cup J)$. 
Hence, the lifting $\wt G\circ\wt F\circ\wt G^{-1}$ 
of the conjugate $F_G:=G\circ F\circ G^{-1}$ is a $C^{\infty}$-smooth diffeomorphism 
that is the product of $\wt G$-conjugates of the involutions $I$ and $\beta$. 
One has $\wt G\circ I\circ \wt G^{-1}=I$, by (\ref{wtg}); $F_G$ is a symplectomorphism, since so are $F$ and $G$; 
\begin{equation}G(s,y)=(\wh s,\wh y)=(\wh s(s,0)+O(y), g(s,0)y+O(y^2)),\label{gsyas}\end{equation} 
by diffeomorphicity.  Substituting (\ref{gsyas}) and (\ref{fbm}) to the expression 
$F_G=G\circ F\circ G^{-1}$ and denoting $(s,0):=G^{-1}(\wh s,0)$, we get 
$$F_G(\wh s,\wh y)=(\wh s+\frac{\partial\wh s}{\partial s}(s,0)w(s)g^{-\frac12}(s,0)
(\wh y)^{\frac12}+O(\wh y), 
\wh y+O(\wh y^{\frac32})).$$
This implies that the conjugate $F_G$ has type (\ref{fbm}) and hence, is  
strongly billiard-like. This proves the proposition.
\end{proof}

 \begin{convention} Let $J=(a,b)\times\{0\}\subset\rr^2_{x,y}$. Let $U\subset\{ y>0\}$ be a domain adjacent to $J$. 
 Let $F:U\cup J\to\rr\times\rr_{\geq0}$ be a map fixing all the points of the interval $J$. 
 Let $\wt h:U\cup J\to\rr_{\geq0}$ be a $C^{\infty}$-smooth $F$-invariant function, i.e., $\wt h(z)=\wt h\circ F(z)$ whenever 
 $z,F(z)\in U\cup J$, and let  
 \begin{equation}\wt h|_J\equiv0, \ \ \frac{\partial\wt h}{\partial y}>0.\label{condh}\end{equation}
 Let $\wt h=const$ denote the foliation by connected components of level curves of the function $\wt h$. This  
 is a $C^\infty$-smooth foliation on $U\cup J$, with $J$ being a leaf. It will 
 be called a {\it foliation by $F$-invariant curves.}
 \end{convention}

\begin{theorem} \label{thm3} For every $C^{\infty}$-lifted strongly billiard-like map $F$ there exists a domain 
$U$  adjacent to $J$ such that $U\cup J$ 
admits a $C^{\infty}$-smooth $F$-invariant function $\wt h$ satisfying (\ref{condh}); thus, 
$\wt h=const$ is a foliation by $F$-invariant curves. 
Moreover,  $U$ can be chosen so that there is an infinite-dimensional family of foliations as above
 with distinct boundary germs.
\end{theorem}
\begin{theorem} \label{addthm3} For every function $\wt h$ from 
 Theorem \ref{thm3},  replacing it by its post-composition with a $C^\infty$-smooth 
 function of one variable (which does not  
 change the foliation $\wt h=const$) one can achieve that there exists a $C^{\infty}$-smooth 
 function $\tau=\tau(s,y)$ 
and a domain $U\subset\{ y>0\}$ adjacent to $J$ such that  $(\tau,\wt h)$ are symplectic 
coordinates on $U\cup J$ and in these coordinates 
\begin{equation}F(\tau,\wt h)=(\tau+\sqrt{\wt h},\wt h).\label{snform}\end{equation}
\end{theorem}

\begin{definition} Let 
 $V\subset \rr\times\rr_{+}\subset\rr^2_{s,y}$ be a domain 
 adjacent to an interval $J=(a,b)\times\{0\}$. A $C^{\infty}$-smooth 
 function $f(s,y)$ on $V\cup J$ is {\it $y$-flat,} if $f(s,0)\equiv0$,  $f(s,y)$ tends to zero with all its partial derivatives, as $y\to0$, and the latter convergence 
 is uniform on compact subsets in the $s$-interval $(a,b)$ for the function $f$ 
 and for each its individual derivative. 
 \end{definition}
 \begin{remark} \label{remfl} In the  conditions of the above definition let $(x,h)$ be new 
 $C^{\infty}$-smooth   coordinates on $V\cup J$ with $h(s,0)\equiv0$. 
 Then each $y$-flat function is $h$-flat 
 and vice versa. This follows from definition.
 \end{remark}

The proof of Theorem \ref{thm3} uses  Marvizi--Melrose result \cite[theorem (3.2)]{mm} stating a formal analogue of Theorem \ref{thm3}: existence of a $F$-invariant 
formal power series $\sum_kh_k(s)y^s$, see Theorem \ref{thmm} below. It implies 
that in appropriate coordinates $(\tau,h)$ the map $F$ takes the form 
$F(\tau,h)=(\tau+\sqrt h+\flat(h), h+\flat(h))$. Here $\flat(h)$ is an $h$-flat function, 
see the above definition. In the coordinates $(\tau,\phi)$, $\phi=\sqrt h$, the lifted 
map $\wt F$  takes the form 
 \begin{equation}\wt F(\tau,\phi)=(\tau+\phi+\flat(\phi), \phi+\flat(\phi)).\label{tauphi}
\end{equation}
We prove  existence  of a $C^{\infty}$-smooth $\wt F$-invariant function $\wt \phi$ with  
$\wt\phi-\phi=\flat(\phi)$ (the next theorem), and then deduce the existence statements in 
Theorems \ref{thm3}, \ref{thm1}, \ref{thm2}.
\begin{theorem} \label{thm33} Let $V\subset\rr_{\tau}\times(\rr_+)_{\phi}$ be a domain adjacent to 
an interval $J=(a,b)\times\{0\}$. Let $\wt F:V\cup J\to \rr_{\tau}\times(\rr_{\geq0})_{\phi}$ 
be a $C^{\infty}$-smooth mapping of type (\ref{tauphi}). 
(Here we do not assume any area-preserving 
property.) 

1) There exists a domain $W\subset V$ 
adjacent to $J$ and an $\wt F$-invariant $C^{\infty}$-smooth function $\wt\phi$ on $W\cup J$  
of the type $\wt\phi(\tau,\phi)=\phi+\flat(\phi)$; $\frac{\partial\wt\phi}{\partial\phi}>0$. 

2) For every function $\wt\phi$ as above one can shrink the domain $W$ 
(keeping it adjacent to $J$) so that there exists a $C^{\infty}$-smooth function $\wt\tau(\tau,\phi)=\tau+\flat(\phi)$ 
such that the map $(\tau,\phi)\mapsto(\wt\tau,\wt\phi)$ is a $C^{\infty}$-smooth diffeomorphism on 
$W\cup J$ that conjugates $\wt F$ to the map 
 \begin{equation}\wh F: (\wt \tau,\wt \phi)\mapsto(\wt\tau+\wt\phi,\wt\phi).\label{tauphiob}
\end{equation}

3) There exist continuum of functions $\wt\phi$ satisfying Statement 1) 
 such that  the corresponding foliations $\wt\phi=const$ 
are $C^{\infty}$-smooth on the same subset $W\cup J$ and form an infinite-dimensional family of 
foliations with distinct boundary germs.
\end{theorem}

\subsection{Unique determination of jets. Space of germs of foliations. Non-uniqueness}

\begin{theorem} \label{unget} All the germs of foliations satisfying the statements of any of 
Theorems \ref{thm1}, \ref{thm2}, \ref{thm2closed}, \ref{thm3}, \ref{thm33} at the corresponding boundary curve $\gamma$, $J$
are flatly close to each other near the boundary. That is, they  have the same $n$-jet for every $n$ at each point of the  boundary.
\end{theorem}

The statement of Theorem \ref{unget} follows from Marvizi--Melrose result \cite[theorem (3.2)]{mm} recalled below as Theorem \ref{thmm}. 
For completeness of presentation, we will give a proof of Theorem \ref{unget} in Subsection 2.8. 

Let us now describe the space of jets at $J$ of foliations satisfying the statements of  Theorem \ref{thm3}. 
Recall that there exist  coordinates $(\tau,\wt h)$ in which $J=\{\wt h=0\}$ and $F$ acts  as in  (\ref{snform}):  
$F:(\tau,\wt h)\mapsto(\tau+\sqrt{\wt h},\wt h)$. Fix these coordinates $(\tau,\wt h)$. 
Without loss of generality we consider that $0\in J$. 

Consider the foliation $\wt h=const$. 
Let  $\mcf$ be another  $C^\infty$-smooth foliation by $F$-invariant curves defined on a domain $V$ in the upper half-plane $\{\wt h>0\}$ 
adjacent to $J$ that extends $C^\infty$-smoothly to $J$ with $J$ being a leaf. It is a foliation by level curves 
of an $F$-invariant function $g(\tau,\wt h)=\wt h+\flat(\wt h)$, which follows from Theorem \ref{unget}. We can and will normalize $g$ 
so that 
\begin{equation} g(\tau,\wt h)=\wt h+\flat(\wt h), \  \ g(0,\wt h)\equiv \wt h.\label{gnorma}\end{equation}
\begin{remark} The above normalization can be achieved by replacing the function $g$ by its post-composition with a function $\phi+\flat(\phi)$ 
of one variable $\phi$. 
Each foliation from Theorem \ref{thm3} admits a unique $F$-invariant first integral  $g$  as in (\ref{gnorma}) and vice versa: for every 
 $F$-invariant function $g$ as in (\ref{gnorma}) the foliation $g=const$ satisfies the statement of Theorem \ref{thm3}. 
\end{remark}

\begin{proposition} \label{pmodgerm} 1) For every $C^\infty$-smooth function $g$  on a domain  $V\subset\{\wt h>0\}$ adjacent to $J$ that 
extends $C^\infty$-smoothly to $J$, satisfies (\ref{gnorma}) and is 
 invariant under the mapping $F:(\tau,\wt h)\mapsto(\tau+\sqrt{\wt h},\wt h)$ there exist a $\delta>0$ and a unique 
 $C^\infty$- smooth $\wt h$-flat function $\psi(t,\wt h)$ on $S^1\times[0,\delta)$, 
 $S^1=\rr_{t}\slash\zz$ such that 
 \begin{equation} \psi(0,\wt h)=0,\label{pstar}\end{equation} 
 \begin{equation} g(\tau,\wt h)=\wt h+\psi\left(\frac{\tau}{\sqrt{\wt h}},\wt h\right).\label{gth}\end{equation}
 Here we treat $\psi(t,\wt h)$ as a function of two variables that is 1-periodic in $t$.
 Conversely, for every $\delta>0$ every $\wt h$-flat  function $\psi$ on the cylinder $S^1\times[0,\delta)$ 
 satisfying (\ref{pstar})  corresponds to some function 
 $g$ as above via formula (\ref{gth}), defined on $V\cup J$ with  $J=\rr\times\{0\}$ and $V=\{0<\wt h<\delta\}$. 
 
 2) The analogous statements hold for the map $F:(\tau,\phi)\mapsto(\tau+\phi,\phi)$ and the function 
 \begin{equation} g(\tau,\phi)=\phi+\psi\left(\frac{\tau}{\phi},\phi\right).\label{gth2}\end{equation}
\end{proposition}
\begin{theorem} \label{tmodgerm} Every germ at $J$ of  $C^\infty$-smooth  foliation $\mcf$  by invariant curves under a 
map $F$ of one of the types (\ref{snform}) or (\ref{tauphiob}) is defined by a unique germ 
at $S^1\times\{0\}$ of $C^\infty$-smooth $h$-flat function $\psi(t,h)$, $\psi: S^1\times\rr_{\geq0}\to\rr$, $\psi(0,h)=0$, so that $\mcf$ is the foliation by level curves of the corresponding function $g$ given 
 by (\ref{gth}) or (\ref{gth2}) respectively. 
 Conversely, each germ of function $\psi$ as above defines a germ of foliation as above at $J$. 
 \end{theorem} 

\begin{theorem} \label{tmog} The space of germs  of  foliations  satisfying statements of 
any of Theorems \ref{thm1}, \ref{thm2}, \ref{thm2closed}, 
\ref{thm3}, \ref{thm33} at the corresponding boundary curve $\gamma$ or $J$ is isomorphic 
to the space of $h$-flat germs at $S^1\times\{0\}$ of $C^\infty$-smooth functions $\psi(t,h)$ on $S^1\times\rr_{\geq0}$ 
with $\psi(0,h)=0$. 
\end{theorem} 

 Theorems \ref{unget},  \ref{tmodgerm}, \ref{tmog} and  Proposition \ref{pmodgerm} will be proved in Subsection 2.8.
 In Subsection 2.9 we deduce non-uniqueness statements of  Theorems \ref{thm1}, \ref{thm2}, \ref{thm2closed}, \ref{thm3}, \ref{thm33} 
from Theorem \ref{tmog} and the following proposition, which will be also proved there.
  
\begin{proposition} \label{distgerm} Let $J=(a,b)\times\{0\}$, 
$W\subset\rr_\tau\times(\rr_+)_\phi$ be a domain adjacent to $J$. Let $\wt F:W\cup J\to\rr\times\rr_{\geq0}$ 
be a map, as in 
(\ref{tauphi}).  Any two $\wt F$-invariant foliations (functions, line fields) on $W$ having distinct 
germs at $J$ have distinct germs at each point in $J$. 
The same statement holds for similar objects invariant under 
a $C^{\infty}$-lifted strongly billiard-like map.
\end{proposition}

\subsection{Corollaries on conjugacy  of open billiard maps  near the boundary}

The results stated below and proved in Subsection 2.10 concern (symplectic) conjugacy of billiard maps 
near the boundary. 

Here we deal with a strictly convex oriented $C^\infty$-smooth curve $\gamma\subset\rr^2$ that is not closed: 
parametrized by an interval. 
We consider that it is positively oriented as  the boundary of its convex side.  Let 
us first consider that  $\gamma$ 
goes to infinity in both directions and bounds a convex open billiard.  
By $\wt\gamma$ we denote the family of its orienting unit tangent vectors; $\wt\gamma$ lies in the space $\Gamma$, 
which is the unit tangent bundle of the ambient plane restricted to $\gamma$. 
Let $s$ be a natural length parameter of the curve $\gamma$. 
 Let $\Gamma^0_+\subset\Gamma$ be the open subset adjacent to $\wt\gamma$ defined in Subsection 1.2. It  
 lies in the space of 
pairs $(s,v)$ where $s\in\gamma$ and $v\in T_s\rr^2$ is a unit 
vector directed to the convex side from the curve $\gamma$. Recall that $\phi=\phi(v)$ denote the angle 
between the vector $v$ and the unit tangent vector $\dot\gamma(s)$. Let $(a,b)=(a_\gamma,b_\gamma)\subset\rr$ denote the 
length parameter interval parametrizing $\gamma$.   In the coordinates $(s,\phi)$ the curve 
$\wt\gamma$ is the interval $J=J_\gamma=(a,b)\times\{0\}$, and $\Gamma^0_+$ is 
a domain in $\rr\times\rr_+$ adjacent to $J$.  
Set $y=1-\cos\phi$, see (\ref{defy}).  Recall that the billiard map $\mathcal T_\gamma$
acting by reflection from $\gamma$ of the above unit vectors is a $C^\infty$-smooth 
diffeomorphism defined on $\Gamma^0_+\cup\wt\gamma$. 
In the coordinates $(s,y)$ it is a symplectic map: a $C^\infty$-lifted strongly billiard-like map defined on 
$\mathcal V\cup J$, where $\mathcal V\subset \rr_s\times(\rr_+)_y$ is a domain adjacent to $J$. 

The above statements remain valid in the case, when the curve $\gamma$ 
 in question is a  
subarc (parametrized by interval but not necessarily infinite) of a strictly convex $C^\infty$-smooth curve; $\gamma$ also may be an immersed curve. 

\begin{definition} Let $\gamma_1,\gamma_2\subset\rr^2$ be strictly convex $C^\infty$-smooth  planar curves parametrized by intervals 
(they are allowed to be immersed), positively oriented as boundaries of their convex sides. Let $J_{\gamma_i}\subset\rr\times\{0\}$, $i=1,2$, be the corresponding intervals 
defined above. We say that the billiard maps $\mathcal T_{\gamma_i}$ are {\it $C^\infty$-smoothly conjugated 
near the boundary in the $(s,\phi)$- ($(s,y)$-) coordinates} if there exist domains $U_i$ in $\rr_s\times(\rr_+)_\phi$ 
(respectively, in $\rr_s\times(\rr_+)_y$) adjacent to $J_{\gamma_i}$ and a $C^\infty$-smooth diffeomorphism 
$H:U_1\cup J_{\gamma_1}\to U_2\cup J_{\gamma_2}$ conjugating the billiard maps,  
$H\circ\mct_{\gamma_1}\circ H^{-1}=\mct_{\gamma_2}$. In the case, when the billiard maps are conjugated in the 
$(s,y)$-coordinates, and the conjugating diffeomorphism $H$ is a symplectomorphism, we say that they are 
$C^\infty$-smoothly {\it  symplectically conjugated near the boundary.}
\end{definition}

\begin{remark} \label{rsysf} Smooth conjugacy of billiard maps near the boundary in the coordinates $(s,y)$ 
implies their smooth conjucacy in the coordinates $(s,\phi)$. This follows from the fact that 
for every two intervals $J_1,J_2\in\rr_s\times\{0\}$ and every two domains $U_1,U_2\subset\rr_s\times(\rr_+)_y$ 
adjacent to $J_1$ and $J_2$ respectively every diffeomorphism $H:U_1\cup J_1\to U_2\cup J_2$ 
lifts to a diffeomorphism of the corresponding domains in the $(s,\phi)$-coordinates (taken together with 
adjacent intervals $J_i$). The latter statement follows from \cite[lemma 3.1]{gpor} applied to the second 
component of the diffeomorphism $H$.
\end{remark}

The results stated below on conjugacy of billiard maps near the boundary 
are corollaries of Theorems \ref{addthm3} and 
\ref{thm33} on normal forms of $C^\infty$-lifted strongly  billiard-like maps and their liftings.

\begin{theorem} \label{thconj1} Let $\gamma_1$, $\gamma_2$ be strictly convex $C^\infty$-smooth one-dimensional submanifolds in $\rr^2$ parametrized by intervals (thus, going to infinity in both 
directions) and positively oriented as boundaries of their convex sides. Let in addition, 
the curves $\gamma_i$ have finite asymptotic tangent lines at infinity: as $x\in\gamma_i$ tends to infinity 
(in each direction), the tangent line $T_x\gamma_i$ 
converges to a finite line. Then the corresponding billiard maps are $C^\infty$-smoothly conjugated near 
the boundary in $(s,y)$- (and hence, in $(s,\phi$)-) coordinates.
 \end{theorem}

 \begin{theorem} \label{thconj2} The statement of Theorem \ref{thconj1} on conjugacy of 
 billiard maps corresponding to $C^\infty$-smooth strictly convex curves $\gamma_i$ 
  remains valid in the case, when each $\gamma_j$ is either a submanifold going to infinity in both directions, 
 or a (may be immersed)  subarc of an (immersed) $C^\infty$-smooth curve, and the two following 
 statements hold:
 
 1) as the length parameter $s$ of the curve $\gamma_j$ goes to an endpoint of the length parameter interval, 
 the corresponding point of the curve $\gamma_j$ tends either to a finite limit (endpoint of $\gamma_j$) 
 where $\gamma_j$ is $C^2$-smooth, or to infinity;
 
 2) in the latter case, when the limit is infinite, the tangent line $T_s\gamma_j$ has a finite limit: 
 a finite asymptotic tangent line.  
  \end{theorem}
  
  \begin{remark} V. Kaloshin and C.E. Koudjinan \cite{kalkoudj} proved continuous 
  conjugacy near the boundary of  two  billiard maps corresponding to two arbitrary ellipses. 
  For any two ellipses with two appropriate points deleted in each of them they have also proved smooth conjugacy 
  of the corresponding billiard maps on open domains adjacent to the corresponding boundary intervals $J$  
  in the $(s,\phi)$-coordinates.  
  \end{remark}
   
 Below we state a more general result and provide a sufficient condition of symplectic conjugacy of 
 billiard maps in the coordinates $(s,y)$. 
 To this end, let us recall the following definition. 
 
\begin{definition} Let $\gamma$ be a $C^2$-smooth oriented planar curve, and let $s$ be its length parameter 
defining its orientation. Let $I_\gamma=(a_\gamma,b_\gamma)\subset\rr_s$ 
denote the length parameter interval parametrizing $\gamma$. 
Let $\kappa=\kappa(s)$ denote the geodesic curvature of the curve $\gamma$ as a function of $s$. 
The {\it Lazutkin length} of the curve $\gamma$ is the integral 
\begin{equation}\mcl(\gamma):=\int_{a_\gamma}^{b_\gamma}\kappa^{\frac23}(s)ds,\label{lazl}\end{equation}
see \cite[formula (1.3)]{laz}. 
(While the length parameter interval is  defined up to translation, the integral $\mcl(\gamma)$ is 
uniquely defined.)
\end{definition}

\begin{theorem} \label{thconj3} Let $\gamma_1$, $\gamma_2$ be two strictly convex $C^\infty$-smooth 
(may be immersed) planar curves, parametrized by intervals and positively oriented as local boundaries of their convex sides. The 
corresponding billiard maps are $C^\infty$-smoothly conjugated near the boundary 
in $(s,y)$-coordinates,  if and only if one of the two following conditions holds: 

i) either both Lazutkin lengths $\mcl(\gamma_i)$ are finite;

ii) or  both Lazutkin lengths $\mcl(\gamma_i)$ are infinite and the improper integrals defined them 
are 

- either both infinite in both directions;

- or both infinite in one and the same direction (with respect to the orientations of the curves $\gamma_i$) 
and both finite in the other direction.

The same criterium also holds for $C^\infty$-smooth conjugacy near the boundary in $(s,\phi)$-coordinates.
\end{theorem}

\begin{theorem} \label{thconj4} Let in the conditions of Theorem \ref{thconj3},  some of 
conditions i) or ii) hold.   Then the billiard maps  are $C^\infty$-smoothly {\bf symplectically} conjugated near 
the boundary, 
if and only if  the Lazutkin lengths of the curves $\gamma_j$  are either both finite and equal, 
 or both infinite and the above condition ii) holds. 
\end{theorem}

Theorems \ref{thconj1} and \ref{thconj2} will be deduced from Theorem \ref{thconj3} using the following 
propositions on $C^\infty$-lifted strongly billiard-like maps and lemma on curves with asymptotic line at infinity. 

\begin{proposition} \label{proform} Let $F(s,y)=(s+w(s)\sqrt y+O(y),y+O(y^{\frac32}))$ be a $C^\infty$-lifted 
strongly billiard-like map, see (\ref{fbm}), defined on $U\cup J$, where $J=(a,b)\times\{0\}$ and 
$U\subset\rr\times\rr_+$ is a domain adjacent to $U$. Let $H(s,y)=(H_1(s,y),H_2(s,y))$ be a $C^\infty$-smooth diffeomorphism 
of the domain with boundary $U\cup J$ onto its image in $\rr\times\rr_{\geq0}$ 
 that conjugates $F$ with its normal form 
$\La:(t,z)\mapsto(t+\sqrt z,z)$, i.e., $H\circ F\circ H^{-1}=\La$.   Fix a $s_0\in(a,b)$. 

1) The diffeomorphism $H$ is orientation-preserving, $H(J)\subset\rr\times\{0\}$, and the restriction $H_1(s,0)$ to 
$J$ of  its first component is an increasing function.

 2) If $H$ is symplectic, then, up to additive constant, 
\begin{equation}H_1(s,0)=t_L(s):=\int_{s_0}^sw^{-\frac23}(s)ds.\label{hso}\end{equation} 

3) If $H$ is not necessarily symplectic, then 
\begin{equation}H_1(s,0)=\alpha t_L(s)+\beta \text{ for some  } \alpha>0 \text{ and } 
\beta\in\rr.\label{hso2}\end{equation} 
\end{proposition}

\begin{proposition} \label{proform2} Let $F$, $U$, $J$ be the same, as in Proposition \ref{proform}. 
Let $\wt F$ be the lifting of the  map $F$ to the coordinates 
$(s,\psi)$, $\psi^2=y$, which is a $C^\infty$-smooth diffeomorphism 
defined on $\wt U\cup J$, $\wt U=\{(s,\psi) \ | \ (s,\psi^2)\in U, \ \psi>0\}$. Let $\wt H$ be a $C^\infty$-smooth diffeomorphism defined on 
$\wt U\cup J$ conjugating $\wt F$ with the diffeomorphism $\wt\La:(t,\wt z)\mapsto(t+\wt z,\wt z)$: 
$\wt H\circ\wt F\circ\wt H^{-1}=\wt\La$. Then $\wt H(J)\subset\rr\times\{0\}$, and the first component of the map 
$\wt H$ satisfies  (\ref{hso2}).
\end{proposition}

\begin{lemma} \label{lazconv} Let a $C^\infty$-smooth strictly convex planar curve $\gamma$ 
go to infinity in some direction, 
and let it have a finite asymptotic tangent line at infinity in this direction (in the same sense, as in 
Theorem \ref{thconj1}). Then the improper integral (\ref{lazl}) defining the Lazutkin length converges  in the 
given direction. 
\end{lemma}

\begin{remark} For a $C^\infty$-smooth strictly convex 
planar curve going to infinity, existence of finite asymptotic tangent line 
is not a necessary condition for convergence of the improper integral (\ref{lazl}) defining the Lazutkin length. 
Namely, consider the graph $\{ y=x^r\}\subset[1,+\infty)\times[1,+\infty)$, $r>1$. One has 
\begin{equation}\int\kappa^{\frac23}(s)ds<+\infty, \ \text{ if and only if } \  r>2.\label{ifr>2}\end{equation}
Indeed,  \ $ds=\sqrt{1+r^2x^{2(r-1)}}dx$, \ $\kappa(s(x))=\frac{r(r-1)x^{r-2}}{(1+r^2x^{2(r-1)})^{\frac32}}$, \ 
see (\ref{implaz}), 
$$ \kappa^{\frac23}(s(x))ds=(r(r-1))^{\frac23}
\frac{x^{\frac23(r-2)}}{\sqrt{1+r^2x^{2(r-1)}}}dx\simeq(r(r-1))^{\frac23}x^\nu dx, \ \nu=-\frac{r+1}3.$$
Therefore, the integral in (\ref{ifr>2}) converges, if and only if $\nu<-1$, i.e., $r>2$. 
In the case of parabola $\{ y=x^2\}$ the integral (\ref{ifr>2}) diverges. 
\end{remark}

\subsection{Plan of the proof of main results}

In Subsection 2.1 we recall the above-mentioned Marvizi -- Melrose result \cite[theorem 3.2]{mm} (with proof)  yielding $C^{\infty}$-smooth coordinates in which a $C^\infty$-lifted strongly billiard-like map $F$ takes the form
 $F(\tau,h)=(\tau+\sqrt h+\flat(h), h+\flat(h))$. It implies that the lifted map $\wt F$, written 
 in the coordinates $(\tau,\phi)$, $\phi=\sqrt h$, takes form (\ref{tauphi}). 

Theorem \ref{thm33}, Statement 1) will be proved in  Subsections 2.2--2.4. 
To do this, first in Subsection 2.2 we construct a fundamental domain for the map 
$\wt F$ (a curvilinear sector $\Delta$ with vertex at a point in $J$) and an $\wt F$-invariant function $\wt\phi$ defined on a bigger sector that  is $\phi$-flatly close to $\phi$ 
on the latter bigger sector. Then in Subsection 2.3 we construct its $\wt F$-invariant extension  along the $\wt F$-orbits and show 
that it is well-defined on a domain adjacent to $J$. In Subsection 2.4 we prove 
that thus extended function $\wt\phi$ is $C^{\infty}$-smooth 
and $\phi$-flatly close to $\phi$. This will prove Statement 1) of Theorem \ref{thm33}. Its Statement 2) on normal 
form will be proved in Subsection 2.5. 

The existence statement in Theorem \ref{thm3} will be deduced from Statement 1) of Theorem \ref{thm33} 
in Subsection 2.6, where we will also prove  Theorem \ref{addthm3}. 
Existence  in Theorems \ref{thm1}, \ref{thm2}  and \ref{thm2closed} will be proved in Subsection 2.7. 
The results from Subsection 1.4 on jets and space of germs of foliations will be proved in Subsection 2.8. 
Proposition \ref{distgerm} and 
non-uniqueness statements in main theorems    will be proved in Subsection 2.9. 

The results of Subsection 1.5 on conjugacy of billiard maps near the boundary will be 
proved in Subsection 2.10.

\subsection{Historical remarks} 
The Birkhoff Conjecture was first stated in print by H. Poritsky \cite{poritsky}, who proved it under additional condition that 
for any two nested closed caustics the smaller one is a caustic of the billiard played in the bigger one; the same result was later obtained 
in \cite{amiran}. One of the most famous results 
on the Birkhoff Conjecture is due to M. Bialy \cite{bialy}, who proved that if the phase cylinder of the billiard is foliated by non-contractible invariant 
closed curves, then the billiard boundary is a circle; see also another proof in \cite{wojt}. Recently V. Kaloshin and A. Sorrentino 
proved that any integrable deformation of an ellipse is an ellipse \cite{kalsor}. Very recently M. Bialy and A. E. Mironov proved 
the Birkhoff Conjecture for centrally-symmetric billiards having a family of closed caustics that extends up to a caustic tangent to four-periodic orbits 
\cite{bm6}. For a detailed survey of the Birkhoff Conjecture see \cite{kalsor, KS18, bols, bm, bm6, 
gl2, tab08} and references therein. 

Existence of a Cantor family of closed caustics in every strictly convex bounded planar billiard with sufficiently smooth boundary was 
proved by V. F. Lazutkin \cite{laz} using KAM type arguments. 

R. Melrose proved that for every $C^\infty$-smooth germ $\gamma$ of strictly convex planar curve there exists a germ of $C^\infty$-smooth 
foliation by caustics of the billiard  played on $\gamma$, with $\gamma$ being a leaf \cite[p.184, proposition (7.14)]{melrose1}. 

S. Marvizi and R. Melrose 
have shown that the billiard ball map $\mct$ in a planar domain bounded by a $C^\infty$-smooth strictly convex closed curve $\gamma$ 
always  has an {\it asymptotic first integral} on a domain with boundary in the space of oriented lines: a domain adjacent to the 
family of tangent lines to $\gamma$. 
Namely, there exists a $C^\infty$-smooth function $F$ 
on the closure of a domain as above such that 
the difference $F\circ T-F$ is $C^\infty$-smooth there,  and it is flat at the points of the family of tangent lines to $\gamma$; 
see \cite[theorem (3.2)]{mm}; see also  statement of their result in Theorem \ref{thmm} below.

(Strongly) billiard-like maps were introduced and studied in \cite{gpor}, where results on their dynamics were applied to curves with 
Poritsky property.  

V. Kaloshin and E.K.Koudjinan proved that for a  non-integrable billiard bounded by 
a strictly convex closed curve, the Taylor coefficients of the normalized Mather 
$\beta$-function are invariant under $C^\infty$-conjugacies \cite{kalkoudj}. They also obtained a series 
of results on conjugacy of elliptic billiard maps, showing in particular that global topological 
conjugacy implies similarity of underlying ellipses.

\section{Construction of foliation by invariant curves. Proofs of main results}

\subsection{Marvizi--Melrose construction of an "up-to-flat" first integral}

Here we recall the following Marvizi--Melrose theorem  with proof. Though it was stated in \cite{mm} for 
billiard ball maps, its statement and proof remain valid for $C^\infty$-lifted strongly billiard-like maps. 
\begin{theorem} \label{thmm} \cite[theorem (3.2)]{mm}. 
1) Let $V\subset(a,b)\times\rr_{>0}\subset\rr^2_{s,y}$ be a domain adjacent to 
the interval $J:=(a,b)\times\{0\}$. Let $F:V\cup J\to\rr\times\rr_{\geq0}$ be 
a $C^{\infty}$-lifted 
strongly billiard-like map. There exist a domain $W\subset V$ adjacent to $J$ 
and a real-valued $C^{\infty}$-smooth function $h:W\cup J\to\rr_{\geq0}$, 
$h|_J\equiv0$, $\frac{\partial h}{\partial y}|_J>0$, such that the difference $h\circ F-h$ is $C^{\infty}$-smooth 
and $y$-flat. Moreover, 
one can normalize $h$ as above  so that the mapping 
$F$ coincides, up to $y$-flat terms, 
 with the time 1 map of the flow of the Hamiltonian vector 
field with the Hamiltonian function $\frac23 h^{\frac32}$. This normalization determines 
the asymptotic Taylor series $h(s,y)=\sum_{k=1}^{+\infty}h_k(s)y^k$ of the function $h(s,y)$ uniquely. 

2) The analogue of the above statement holds if $J$ is replaced by the coordinate circle 
$S^1=S^1\times\{0\}$, $S^1:=\rr_s\slash\zz$, lying in the cylinder $C:=S^1\times[0,\var)$ equipped with the standard area form and $F$ is a strongly billiard-like map $C\to S^1\times\rr_{\geq0}$.  In this case 
the coefficients $h_k(s)$ of the above normalized series are  1-periodic and $C^{\infty}$-smooth. 

3) Let $h$ be the function normalized as in Statement 1). Let 
$\tau$ denote the time function for the Hamiltonian vector field with the 
Hamiltonian function $h$. In the coordinates $(\tau,h)$  (which are symplectic) 
 the map $F$ takes the form 
\begin{equation} F:(\tau,h)\mapsto(\tau+\sqrt h+\flat(h), h+\flat(h)).\label{taut}\end{equation}
\end{theorem}

\begin{proof} 
The lifting $\wt F(s,z)$, $z=\sqrt y$, of the map $F(s,y)$ is $C^{\infty}$-smooth and 
has the form
\begin{equation} \wt F(s,z)=(s+w(s)z+O(z^2), \ z+\frac{q(s)}2z^2+O(z^3)),
\label{lift}\end{equation}
where $q(s)$ is a $C^{\infty}$-smooth function on $(a,b)$. This follows from (\ref{fbm}) and $C^\infty$-liftedness. 
The map $\wt F(s,z)$ admits an asymptotic Taylor series in $z$. 
The map $F$ has the form  
\begin{equation} F(s,y)=(s+w(s)\sqrt y+O(y), \ y+q(s)y^{\frac32}+O(y^2)),
\label{fform}\end{equation} 
by (\ref{lift}), and 
it admits an asymptotic Puiseux series in $y$ involving powers $0, \frac12, 1, \frac32, 2,\dots$. 
The coefficients of both series are $C^{\infty}$-smooth functions in $s$.  Therefore, the mapping $F$ acts  by the formula $h\mapsto h\circ F$ not only on functions, but also on formal Puiseux series. It transform each power series 
$h=\sum_{k=1}^{+\infty}h_k(s)y^k$ with 
coefficients being $C^{\infty}$-smooth functions on $(a,b)$ 
to a Puiseux series of the above type. 
 Our goal is to find 
an $F$-invariant power series (or equivalently, an $\wt F$-invariant even power series 
$\sum_{k=1}^{+\infty}h_k(s)z^{2k}$) and then to choose its $C^{\infty}$-smooth representative. 
To do this, we use the following formula 
for the function $q(s)$ in (\ref{fform}), see \cite[formula (1.2)]{laz}, 
\cite[formula (7.18)]{gpor}, which follows from area-preserving property: 
\begin{equation} q(s)=-\frac23 w'(s).\label{qw}\end{equation}

Step 1: constructing an even series $\sum_{k=1}^{+\infty}g_k(s)z^{2k}$ 
whose $\wt F$-image is also an even series. 
We construct  its coefficients $g_k$ by induction as follows. 

Induction base: $k=1$. Let us find a function $g_1(s)$ such that the $\wt F$-image of the function 
$g_1(s)z^2$ contains no $z^3$-term. This is equivalent to the statement saying  that the function 
$g_1(s+w(s)z)(z+\frac{q(s)}2z^2)^2$ contains no $z^3$-term, which is in its turn equivalent 
to the differential equation
$$g_1'(s)w(s)+q(s)g_1(s)=0, \ \ q(s)=-\frac23w'(s),$$
which has a unique solution $g_1(s)=w^{\frac23}(s)$ up to constant factor. 
(Note that $w^{\frac23}(s)y$ is a well-known function:   
 the second  Lazutkin coordinate \cite{laz, mm}.) 
 
 Induction step in the case, when $J=(a,b)\times\{0\}$ is an interval. 
  Let we have already found an even Taylor polynomial 
 $G_{n-1}(s,z):=\sum_{k=1}^{n-1}g_k(s)z^{2k}$, $n\geq2$, such that the asymptotic Taylor series in $z$ of the function $G_{n-1}\circ\wt F$ contains no odd powers of $z$ 
 of degrees no greater than $2n-1$. Let us construct $g_{n}(s)$, 
 set  $G_n(s,z):=\sum_{k=1}^{n}g_k(s)z^{2k}$, so that 
 \begin{equation} G_n\circ\wt F-G_n \text{ contains no } z^{2n+1}-\text{term}.
 \label{gnterm}\end{equation}
Note that $G_n\circ\wt F-G_n$ obviously cannot contain odd powers of degrees less 
 than $2n$. Let $b(s)z^{2n+1}$ denote the degree $2n+1$ term in the Taylor series 
 of the function $G_{n-1}\circ\wt F$.  Condition (\ref{gnterm}) 
 is equivalent to the differential equation
 \begin{equation} g_n'(s)w(s)-\frac{2n}3w'(s)g_n(s)=-b(s),\label{gnb}\end{equation}
which  always has a solution $g_n(s)$ well-defined on the  interval $(a,b)$.  

Step 2. Constructing an $\wt F$-invariant series. The mapping $\wt F$ is the 
product $I\circ\beta$ of two involutions: $I(s,z)=(s,-z)$ and $\beta$. Let 
$g:=\sum_{k=1}^{+\infty}g_k(s)z^{2k}$ be the series constructed on Step 1. 
One has 
\begin{equation} g\circ\wt F=(g\circ I)\circ\beta=g\circ\beta,\label{gcircb}\end{equation}
since the series $g$ is even. The series (\ref{gcircb}) is  even (Step 1). Hence, the series 
$$t:=g+g\circ\beta$$
is even  and $\beta$-invariant by construction. Therefore, it is $\wt F$-invariant. 
Its first coefficient is equal to $2g_1(s)=2w^{\frac23}(s)>0$, by construction. 
We denote the $\wt F$-invariant series thus constructed by $t:=\sum_{k=1}^{+\infty}t_k(s)z^{2k}$. 

Step 3: symplectic coordinates and normalization. Let $t(s,y)$ be a 
 function representing the series $\sum_{k=1}^{+\infty}t_k(s)y^{k}$, which is  
 obtained  from the latter series (given by Step 2) by the variable change $y=z^2$.  
It is defined on a domain $W$ adjacent to $J$ and $C^{\infty}$-smooth 
 on $W\cup J$; $t|_J\equiv0$, $\frac{\partial t}{\partial y}|_J>0$.  Let $H_t$ denote the 
 corresponding Hamiltonian vector field. Fix an arbitrary $C^{\infty}$-smooth 
 function $\theta$ such that $d\theta(H_t)\equiv1$, $\theta|_{s=0}=0$: a time function 
for the vector field $H_t$. Then $(\theta,t)$ are symplectic coordinates for the form $\omega=dx\wedge dy$: $\omega=d\theta\wedge dt$. Shrinking $W$ (keeping it adjacent 
to $J$) we can and will consider that they are global coordinates on $W\cup J$.
 The difference $t\circ F-t$ is $t$-flat, by construction, and hence, so is 
 $dF(H_t)-H_t$. Therefore, in the coordinates $(\theta,t)$ 
the symplectic map $F$ takes the form
\begin{equation}F:(\theta,t)\mapsto(\theta+\xi(t), t) + \flat(t).\label{xit}
\end{equation}
In the new coordinates $(\theta,t)$ the map $F$ is $C^{\infty}$-lifted strongly billiard-like,  
as in the old coordinates $(s,y)$, by Proposition \ref{classinv}. 
\medskip

{\bf Claim 1.} {\it The function $\xi(t)$  in (\ref{xit}) has the form 
$\xi(t)=\sqrt t \psi(t)$, where $\psi(t)$ is a 
$C^{\infty}$-smooth 
function on a segment $[0,\var]$, $\var>0$, $\psi\geq0$, $\psi(0)>0$.}

\begin{proof}  Let  $\wt F$ denote the lifting of the map $F$ to the coordinates 
$(\theta,\zeta)$, $\zeta=\sqrt t$. One has $\wt F=I\circ\beta$, where 
$I(\theta,\zeta)=(\theta,-\zeta)$ and $\beta$ is an involution, 
$\beta(\theta,0)\equiv(\theta,0)$. The involution $\beta$ takes the form 
\begin{equation}\beta(\theta,\zeta)=(\theta+r(\zeta),-\zeta)+\flat(\zeta), \ \ 
r(\zeta)=\xi(\zeta^2) \text{ for } \zeta>0.\label{betat}\end{equation}
The function $r(\zeta)$ should be  $C^{\infty}$-smooth, as is $\beta$, 
and $r'(0)>0$ (strong billiard-likedness). 
The condition saying that $\beta$ is an involution 
implies that $r(\zeta)+r(-\zeta)=\flat(\zeta)$. This in its turn implies that  
$r(\zeta)=\zeta \psi(\zeta^2)+\flat(\zeta)$, 
where $\psi$ is a $C^{\infty}$-smooth function;   $\psi(0)=r'(0)>0$. 
This together with (\ref{betat}) implies the statement of the claim.
\end{proof} 

We have to find  a function $h(s,y)$, $h(s,0)\equiv0$, 
such that the Hamiltonian vector field with the 
Hamiltonian function $\frac23h^{\frac32}$ coincides with $\xi(t)\frac{\partial}{\partial\theta}$.  
This function will satisfy the normalization statement of Theorem \ref{thmm}, part 1), 
by construction. We are looking for it as a function depending only on $t$: 
$h(s,y)=v(t)$. The above  Hamiltonian vector field is then equal to 
$\sqrt{v(t)}v'(t)\frac{\partial}{\partial\theta}$. Thus, we have to solve the equation
$$v^{\frac12}(t)v'(t)=\xi(t)=\sqrt t \psi(t), \ v(0)=0.$$
Its solution $v(t)$ is given by the formula
$$v(t)=\left(\frac32\int_0^t\sqrt p\psi(p)dp\right)^{\frac23}.$$
This is a $C^{\infty}$-smooth function, by construction and smoothness of the function 
$\psi(t)$. One has 
$\frac{\partial h}{\partial y}|_J>0$, since $v'(0)=\psi^{\frac23}(0)>0$ and 
$\frac{\partial t}{\partial y}(s,0)=2g_1(s)=2w^{\frac23}(s)>0$, by construction. 
 Uniqueness of the Taylor series in $y$ of the function $h(s,y)$ 
satisfying the above Hamiltonian vector field 
statement (up to flat terms) follows directly,  as in \cite[p.383]{mm}. Statement  1) of 
Theorem \ref{thmm} is proved. Statement 3) follows immediately from Statement 1), 
since in the coordinates $(\tau,h)$, see Statement 3), the Hamiltonian  
field with the Hamiltonian function $\frac23h^{\frac32}$ is equal to $(\sqrt h, 0)$. 
 Statement 2) (case, when $J$ is a circle and 
$F$ is defined on a cylinder bounded by $J$) says 
that the Taylor coefficients of the series in $y$ of the function $h(s,y)$ 
are well-defined functions on the circle $J$. This follows from its the above 
uniqueness statement (which holds locally, in a neighborhood of every point $(s_0,0)\in J$). Theorem \ref{thmm} is proved.
\end{proof}

\subsection{Step 1. Construction of an invariant function on a neighborhood of fundamental domain} 

Here we give the first step of the proof of Theorem \ref{thm33}. We consider 
a fundamental sector $\Delta$ for the map $\wt F$ that is bounded by the segment 
$K=[0,\frac\eta2]$ of the $\phi$-axis, by its $\wt F$-image and by the straightline segment 
connecting their ends. We construct an 
$\wt F$-invariant function  $\wt\phi$ that is $\phi$-flatly close to $\phi$ on 
a sectorial neighborhood $S_{\chi,\eta}$ of  
$\overline\Delta\setminus\{(0,0)\}$. 
See Fig. 5.
\begin{figure}[ht]
  \begin{center}
   \epsfig{file=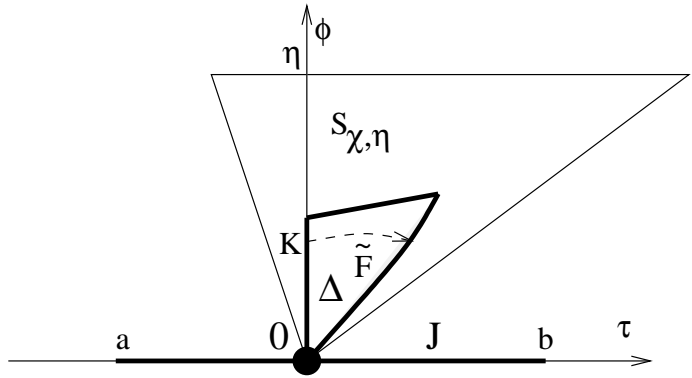}
    \caption{The fundamental domain 
    $\Delta$ and its sectorial neighborhood $S_{\chi,\eta}$.}
    \label{fig:03}
  \end{center}
\end{figure}

Without loss of generality we consider that the $\tau$-interval contains the origin: 
$a<0<b$. Fix a  number $\chi$, $0<\chi<\frac12$. Consider the sectors 
\begin{equation} S_{\chi}=\{ -\chi\phi<\tau<(1+\chi)\phi\}\subset 
\rr_\tau\times(\rr_+)_{\phi},
\label{secd}\end{equation}
 $$S_{\chi,\eta}:= S_{\chi}\cap\{0<\phi<\eta\}$$ 
The domain  $S_{\chi,\eta}$ 
will be the above-mentioned neighborhood of  fundamental sector, where 
we construct an $\wt F$-invariant function. 

\begin{proposition} \label{vi} For every $\chi\in(0,\frac12)$ and $\eta>0$ small enough 
dependently on $\wt F$ and $\chi$ the following statements hold.

(i) The maps $\wt F^{\pm1}$, $\wt F^{\pm2}$ are well-defined on $S_{\chi,2\eta}$.

(ii) The domains $S_{\chi,2\eta}$ and $\wt F^2(S_{\chi,2\eta})$ are disjoint; 
the latter lies on the right from the former.

(iii) The segment $K:=\{0\}\times[0,\frac\eta2]\subset\rr^2_{\tau,\phi}$ and its image $\wt F(K)$ 
intersect just by the origin; $\wt F(K)$ lies on the right from $K$. The 
domain $\Delta\subset S_{\chi,2\eta}$ bounded by $K$, $\wt F(K)$  
and the straightline segment connecting the endpoints of the arcs $K$ and $\wt F(K)$ 
distinct from $(0,0)$ is a fundamental domain for the map $\wt F$. See Fig. 5.
\end{proposition}
\begin{proof} One has 
\begin{equation}d\wt F(0,0)=\left(\begin{matrix} 1 & 1\\ 0 & 1\end{matrix}\right).
\label{diffo}\end{equation}
The latter differential sends each line $\{\tau=\zeta\phi\}$ to the line 
$\{\tau=(\zeta+1)\phi\}$. 
This  implies that  for every $\eta>0$ small enough statements (i)--(iii) hold.
\end{proof}

\begin{proposition} \label{vii} For every $\chi\in(0,\frac12)$ and $\eta>0$ small enough 
dependently on $\wt F$ and $\chi$ 
there exists a $C^{\infty}$-smooth and $\wt F$-invariant function $\wt\phi(\tau,\phi)$ 
on $S_{\chi,\eta}$  such that the difference $\wt\phi(\tau,\phi)-\phi$ is $\phi$-flat on $S_{\chi,\eta}$: that is, 
 tends to zero with all its partial derivatives, 
as $(\tau,\phi)\in S_{\chi,\eta}$ tends to zero. 
\end{proposition}

 \begin{proof} Let $\nu:S_{\chi}\to\rr$ denote the function 
$$\nu:=\frac{\tau}\phi,$$
whose level curves  are  
 lines through the origin. The interval of values of the function $\nu$ on $S_{\chi}$ is 
 $M:=(-\chi,1+\chi)$. Fix a 
 \begin{equation} \sigma>0, \ 2\sigma<\frac12-\chi.\label{sichi}\end{equation} 
 Consider the covering of  the  interval $M$ by the intervals 
 $$(-\chi,\frac12+\sigma), \ \ (\frac12-\sigma,1+\chi)$$ 
 and a corresponding $C^\infty$-smooth partition of unity 
 $\rho_1$, $\rho_2$: $M\to\rr$, 
 \begin{equation}\rho_1\equiv1 \text{ on } (-\chi,\frac12-\sigma); \ \rho_2\equiv 1 \text{ on } 
 (\frac12+\sigma,1+\chi);\label{partun}\end{equation}
 $$ 
 \ \ \rho_1,\rho_2\geq0, \ \rho_1+\rho_2\equiv1 \text{ on } M=(-\chi, 1+\chi).$$
 Set 
 \begin{equation}\wt\phi(x):=\rho_1(\nu(x))\phi(x)+\rho_2(\nu(x))\phi\circ \wt F^{-1}(x)\label{wtphi}\end{equation}
 $$
 =\phi(x)+\rho_2(\nu(x))(\phi\circ\wt F^{-1}(x)-\phi(x)).$$
 \begin{proposition} \label{pinvf}  For every fixed $\chi\in(0,\frac12)$, $\sigma\in(0,\frac12(\frac12-\chi))$  and every $\eta$ small enough 
 (dependently on $\chi$ and $\sigma$) the function $\wt\phi$ given by (\ref{wtphi}) is well-defined 
 on $S_{\chi,\eta}$ and $\wt F$-invariant: if $x, \wt F(x)\in S_{\chi,\eta}$, 
 then $\wt\phi(\wt F(x))=\wt \phi(x)$. It is $C^{\infty}$-smooth, and the difference 
 $\wt\phi(x)-\phi(x)$ is  $\phi$-flat on $S_{\chi,\eta}$. 
 \end{proposition}
 \begin{proof} Recall that $\wt F$ satisfies asymptotic formula (\ref{tauphi}):
 $$\wt F(\tau,\phi)=(\tau+\phi+\flat(\phi), \phi+\flat(\phi)).$$
 Well-definedness and $C^{\infty}$-smoothness of the function 
 $\wt\phi$ on $S_{\chi,\eta}$ for small $\eta$ are obvious. Its $\phi$-flatness on 
 $S_{\chi,\eta}$ follows from formula (\ref{wtphi}), 
 $\phi$-flatness of the difference $\phi\circ\wt F-\phi$, see (\ref{tauphi}),  
 and the fact that the function $\nu(\tau,\phi)=\frac\tau\phi$ has partial derivatives 
 of at most polynomial growth in $\phi$, as $(\tau,\phi)\to0$ along the 
 sector $S_{\chi,\eta}$.  Let us prove $\wt F$-invariance, whenever 
 $\eta$ is small enough. 
 For every $\delta>0$ and  every $\eta>0$ small enough 
 (dependently on $\delta$)  the inclusion $x,\wt F(x)\in S_{\chi,\eta}$ implies 
 that $x\in\{-\chi\phi<\tau<(\chi+\delta)\phi\}$, see (\ref{diffo}). 
 Choosing $\delta<\sigma$, we get that the latter  sector lies in the sector 
 $\{-\chi\phi<\tau<(\frac12-\sigma)\phi\}$,  since $\chi+\delta<\chi+\sigma<\frac12-\sigma$, 
 see  (\ref{sichi}). Thus, on the latter sector $\rho_1\circ\nu\equiv1$ and $\rho_2\circ\nu\equiv0$, see 
  (\ref{partun}). Hence, $\wt\phi(x)=\phi(x)$,  by (\ref{wtphi}). Similarly applying the 
 above argument "in the inverse time" yields that the inclusion 
 $x,\wt F(x)\in S_{\chi,\eta}$ implies that 
 $\wt F(x)$ lies in the sector  $\{(1-\chi-\delta)\phi<\tau<(1+\chi)\phi\}$. The latter sector, and hence, 
 $\wt F(x)$ lie in the sector $\{(\frac12+\sigma)\phi<\tau<(1+\chi)\phi\}$,  since 
 $$1-\chi-\delta>1-\chi-\sigma=1-\chi+\sigma-2\sigma>1-\chi+\sigma-\frac12+\chi=\frac12+\sigma.$$
Therefore,  $\rho_2\circ\nu(\wt F(x))=1$,  
 by (\ref{partun}), and $\wt\phi(\wt F(x))=\phi\circ\wt F^{-1}(\wt F(x))=\phi(x)$, by (\ref{wtphi}). Finally we get that $\wt\phi(x)=\wt\phi\circ\wt F(x)$, and hence $\wt\phi$ is $\wt F$-invariant. 
The proposition is proved.
\end{proof} 

    Proposition \ref{pinvf}  immediately implies the statement of Proposition \ref{vii}.
  \end{proof}
  
  \subsection{Step 2. Extension by dynamics} 
 Here we show that an $\wt F$-invariant function $\wt\phi$ constructed above 
  on a neighborhood  of the fundamental domain $\Delta$ 
  extends along $\wt F$-orbits to an $\wt F$-invariant function on a domain $W$ adjacent to $J=(a,b)\times\{0\}\subset\rr^2_{\tau,\phi}$. The fact that it is $C^\infty$-smooth on $W\cup J$ and coincides with $\phi$ up to 
  $\phi$-flat terms will be proved in the next subsection. It suffices to 
  prove that the function $\wt\phi$ extends as above to a rectangle $(a',b')\times[0,\eta')$ 
  adjacent to arbitrary relatively compact subinterval  $J'=(a',b')\times\{0\}\Subset J$. 
A union of the above rectangles corresponding to an exhaustion of $J$ by a sequence 
  of subintervals $J'$ yields a  domain $W$ adjacent to all of $J$, where 
  the extended function is defined. Therefore, we make the following convention.
  \begin{convention} \label{coun} 
Everywhere below we  identify the interval $J=(a,b)\times\{0\}$ with $(a,b)$ 
and sometimes we denote $J=(a,b)\subset\rr$.  We  consider that $J$ is a finite interval: $a$, $b$ are finite.
 We will consider that 
there exists a $\delta>0$ such that 
$\wt F^{\pm1}$ are diffeomorphisms of the rectangle $J\times[0,\delta)\subset\rr^2_{\tau,\phi}$ onto its images, and the $\phi$-flat terms in  asymptotic formula (\ref{tauphi}) are {\it uniformly 
$\phi$-flat:}  the difference $\wt F(\tau,\phi)-(\tau+\phi,\phi)$ converges to zero 
uniformly in $\tau\in J$, and every its partial derivative 
 (of any order) also converges to zero uniformly, as $\phi\to0$. Indeed, the flat terms in question 
 are uniform on compact subsets in $J$. Hence, one can achieve their 
 uniformity   replacing $J$ by its relatively compact subinterval.
 Under this assumption the above difference and its 
 differential are both uniformly $o(\phi^m)$ in $\tau\in J$ for each individual $m\in\nn$.
 \end{convention}

The next proposition describes asymptotics of two-sided $\wt F$-orbits.

  \begin{proposition} \label{propit}
For every $\eta$ small enough and  $x:=(\tau_0,\phi_0)\in J\times[0,\eta)$ 
 
 a) the iterates $\wt F^j(x)=(\tau_j,\phi_j)$ are well defined for all $j\geq0$, $j\leq N_+$, where 
 $N_+=N_+(x)$ is the maximal number $j$ for which $\tau_j< b$; 
 
 b) the inverse iterates $\wt F^{-j}(x)=(\tau_{-j},\phi_{-j})$ are well-defined for all $j\leq N_{-}$ 
 where $N_-=N_-(x)$ is the maximal number $j$ for which $\tau_{-j}> a$;
 
c)  $\phi_j=\phi_0(1+o(1))$ uniformly in $\tau_0$ and $j\in[-N_-,N_+]$, as $\phi_0\to0$; 

d) the points $\tau_j$ form an asymptotic arithmetic progression: $\tau_{j+1}-\tau_j=
\phi_0(1+o(1))$ uniformly in $\tau_0\in J$ and in 
$j\in[-N_-,N_+-1]$, as $\phi_0\to0$. 
\end{proposition}
\begin{proof}  Consider two lines and segments through $x$: 
$$L_{\pm}(x):=\{\phi=\phi_0\pm\phi_0^{4}(\tau-\tau_0)\}, \ \la_{\pm}:=L_{\pm}
\cap\left(J\times[0,2\eta)\right).$$

{\bf Claim 2.} {\it For   every $x=(\tau_0,\phi_0)\in J\times[0,2\eta)$ 
with $\phi_0$ small enough 

e) the image $\wt F(\la_{\pm})$ is disjoint from $\la_{\pm}$  and lies on its right; 

f) the image $\wt F^{-1}(\la_{\pm})$ is disjoint from $\la_{\pm}$  and lies on its left.

g) the right sector $S_{+}(x)$ bounded by the right subintervals in 
$\la_{\pm}$  with vertex $x$ is $\wt F$-invariant;

h) the left sector $S_{-}(x)$ bounded by the left subintervals in 
$\la_{\pm}$ with vertex $x$ is 
$\wt F^{-1}$-invariant.}

\begin{proof} If $\eta$ is small enough,  then $\wt F^{\pm1}$ are well-defined on 
$J\times[0,3\eta)$. If $\phi_0$ is small enough, then 
 each  $\la_{\pm}$ is projected to all of $J$, 
and the $\phi$-coordinates of all its points are uniformly asymptotically equivalent 
to $\phi_0$ (finiteness of $J$). The map $\wt F$ moves a point  $z=(\tau,\phi)\in\la_{\pm}$ to 
$y:=(\tau+\phi,\phi)$ up to a $\phi$-flat term, which is $o(\phi_0^m)$ for every $m\in\nn$.
 On the other hand, the distance 
of the latter point $y$ to the line $L_{\pm}$ is equal to $\phi\simeq\phi_0$ times the $|\sin|$ of the azimuth of the line $L_{\pm}$. The latter $|\sin|$ is asymptotic to $\phi_0^{4}$, 
and hence, is greater  than $\frac12\phi_0^{4}$, whenever $\phi_0<1$ is small enough. 
Thus, $\dist(y,L_{\pm})\geq \frac13\phi_0^5$. Therefore, adding a term $o(\phi_0^m)$, 
$m\geq5$, 
to $y$ will not allow  to cross $L_{\pm}$, 
and we will get a point lying on the same, right side from the line $L_{\pm}$, as $y$. 
The cases of lines $L_{\mp}$ and inverse iterates are treated analogously. Statements e) and f) are proved. 
They immediately imply statements g) and h).
\end{proof}

 Let  $\eta\in(0,\frac18)$ be small enough so that 
$\wt F$ is defined on the rectangle $\Pi:=J\times[0,3\eta)$
and for every $x\in\Pi$ with $\phi_0=\phi(x)\in[0,2\eta]$ 
the  sector $S_{+}(x)$ contains the points $x_j=\wt F^j(x)$ until they go out 
of $\Pi$ (Claim 2 g)).  The intersection $S_{+}(x)\cap\partial\Pi$ is contained 
in the right lateral side $\{ b\}\times[0,3\eta)$. Therefore, the first 
$j$ for which $x_j$ goes out of 
$\Pi$ is the one for which $\tau(x_j)\geq b$. This proves Statement a) of Proposition 
\ref{propit}. The proof of Statement b) is analogous. For 
 every $x\in\Pi$ with $\phi_0=\phi(x)$ small enough 
the above inclusion $x_j\in S_{+}(x)$ holds for $j=1,\dots,N_+$. It implies 
Statement c) for the above $j$, by the definition of the sector $S_{+}$. 
The proof of Statement c) for $j=-N_-,\dots,-1$ is analogous.  Statement d) follows from
Statement c), since $\tau\circ\wt F(x)-\tau(x)=\phi(x)+\flat(\phi(x))$, see (\ref{tauphi}). 
Proposition \ref{propit} is proved.
\end{proof}

\begin{corollary} \label{1-4}  1) For every $\eta>0$ small enough each point  $x=(\tau_0,\phi_0)\in J\times[0,\frac{2\eta}3)$ 
has two-sided orbit  lying in $J\times[0,\eta)$ and consisting of points $x_j$, 
$j\in[-N_-(x),N_+(x)]$, with $\phi_j\simeq\phi_0$, as $\phi_0\to0$; the latter asymptotics is uniform 
in the above $j$ and in $\tau_0\in J$.  

2) Let $\Delta$ denote the fundamental domain (curvilinear triangle) 
for the map $\wt F$ from 
Proposition \ref{vi}, Statement (iii). Let $\wh\Delta$ denote the complement 
of the closure $\overline\Delta$ to the union of its vertex $(0,0)$ and the opposite side. 
If $\eta>0$ is small enough, then the domain 
$W$ saturated by the above two-sided orbits of points in $\wh\Delta$ lies in 
$J\times[0,\frac{2\eta}3)$ and contains the strip $J\times(0,\frac\eta4)$. 

3) The orbit of each point in $W$ contains 
either a unique point lying in the fundamental domain $\Delta$, 
or two subsequent points lying in its lateral 
boundary curves (glued by $\wt F$). 

4) Each $\wt F$-invariant function $\wt\phi$ on  $\wh\Delta$ 
extends to a unique $\wt F$-invariant function  on   $W$ as a function constant 
along the latter orbits.
\end{corollary}

The corollary follows immediately from Proposition \ref{propit}. Step 2 is done.

\subsection{Step 3. Regularity and flatness. End of proof of  Theorem \ref{thm33}, Statement 1)} 
Here we will prove the following lemma, which will imply Statement 1) of Theorem \ref{thm33}.
\begin{lemma} \label{regflat} Let in Corollary \ref{1-4} the function $\wt\phi$ on $\wh\Delta$ 
be the restriction to $\wh\Delta$ of a $C^{\infty}$-smooth $\wt F$-invariant function defined 
on a neighborhood of  $\wh\Delta$. Let the function 
$\wt\phi(\tau,\phi)-\phi$ be flat on $\wh\Delta$: it tends to zero with all its partial derivatives, as $(\tau,\phi)\in \wh\Delta$ tends to zero. Consider its extension to the above 
domain $W$ from Corollary \ref{1-4}, Statement 4), and let us denote the 
extended function by the same symbol 
$\wt\phi$.  The difference $\wt\phi(\tau,\phi)-\phi$  is $C^{\infty}$-smooth on 
$W\cup J$,  and it is uniformly $\phi$-flat (see Convention \ref{coun}). 
\end{lemma}

\begin{proof} For every point $x=(\tau,\phi)\in W$ there exists a $N=N(x)\in\zz$ 
such that $\wt F^N(x)\in\wh\Delta$. The latter image $\wt F^N(x)$ lies in the definition domain 
of the initial function $\wt\phi$ (which is defined on a neihborhood of $\wh\Delta$), and 
$\wt\phi(x)=\wt\phi_N(x):=\wt\phi(\wt F^N(x))$, by definition. This immediately implies 
$C^{\infty}$-smoothness of the extended function $\wt\phi$ on $W$. Let us prove its 
$\phi$-flatness. This will automatically imply $C^{\infty}$-smoothness at points of 
the boundary interval $J$. To do this, we use the asymptotics
\begin{equation} d\wt F(\tau,\phi)=A+\flat(\phi), \ A=\left(\begin{matrix} 1 & 1\\ 0 & 1
\end{matrix}\right);\label{dwtf}\end{equation}
\begin{equation} N(x)=N(\tau,\phi)=O\left(\frac1\phi\right).\label{asnx}\end{equation}
Here the flat term in (\ref{dwtf}) is uniformly flat, see Convention \ref{coun}.
Formula (\ref{dwtf})   follows from (\ref{tauphi}). Formula (\ref{asnx}) holds, since $N\leq  
N_++N_-=O(\frac1\phi)$, 
which follows from 
Proposition \ref{propit}, Statement d). 

We study the derivatives of the functions $\wt\phi_N-\phi$, $N=N(x)$, at the point $x=(\tau,\phi)$,  as  functions in $x$ with fixed $N$ chosen as above for this concrete $x$. 
To prove uniform flatness, we have to show that all its partial 
derivatives  tend to zero uniformly in $\tau\in J$, as $\phi\to0$. 
We prove this statement for the first derivatives (step 1) and then for the higher derivatives (step 2). 

Without loss of generality everywhere below we consider that $N\geq1$, i.e., 
$x$ lies on the left from the sector $\Delta$: for negative 
$N$ the proof is analogous.

Step 1: the first derivatives. The initial  function $\wt\phi$ defined on a neighborhood 
of the set $\wh\Delta$ is already known to be $\phi$-flat on $\wh\Delta$. The differential  
of the composition $\wt\phi_N=\wt\phi\circ\wt F^N$ at the point $x$, $N=N(x)$, is equal to 
\begin{equation} d(\wt\phi\circ\wt F^N)(x)=d\wt\phi(\wt F^N(x))d\wt F(\wt F^{N-1}(x))\dots d\wt F(x).
\label{diffwt}\end{equation}
 
\begin{proposition} \label{diffto0} For every sequence 
of points $x(k)=(\tau_{0k},\phi_{0k})\in W$ with $\phi_{0k}\to0$, as $k\to\infty$, 
 and  numbers $N_k=N(x(k))\in\nn$ with $\wt F^{N_k}(x(k))\in\wh\Delta$  
 the difference  $d(\wt\phi\circ\wt F^{N_k})(x(k))-d\phi$ tends to zero, as $k\to\infty$.
\end{proposition}
Proposition \ref{diffto0} implies uniform convergence to zero of the first derivatives. 

In its proof (given below) 
 we use the following asymptotics of differential $d\wt F(\wt F^j(x))$ and technical proposition on matrix products. We denote 
 $$M(\tau,\phi):= \text{ the Jacobian matrix of the 
differential } d\wt F(\tau,\phi).$$

\begin{proposition} \label{propxj} Let $x=(\tau_0,\phi_0)\in J\times(0,\frac\eta4)$, 
$x_j=(\tau_j,\phi_j):=
\wt F^j(x)$, $j=0,\dots,N(x)$. For every $m\in\nn$ one has 
 \begin{equation} M(\tau_j,\phi_j)=A+o(\phi_0^m), \ \text{ as } \phi_0\to0; \ A=\left(\begin{matrix} 1 & 1\\ 0 & 1\end{matrix}\right),\label{mjun}\end{equation}
uniformly in $j=1,\dots,N(x)$ and in $\tau_0\in J$ for each individual $m$.
\end{proposition}

\begin{proof} Formula (\ref{mjun}) follows from (\ref{dwtf}) and Proposition \ref{propit}, part  c). 
\end{proof}
 
\begin{proposition} \label{propasm} 
Consider  arbitrary sequences of numbers $\phi_{0k}>0$, $N_k\in\nn$, 
$\phi_{0k}\to0$, $N_k=O(\frac1{\phi_{0k}})$, as $k\to\infty$,  and matrix collections 
$$\mathcal M_k=(M_{1;k},\dots,M_{N_k;k}), \ M_{j;k}\in\gl_2(\rr),$$
\begin{equation} M_{j;k}=A+o(\phi_{0k}^m) \ \text{ for every } \ m\in\nn; \ \ 
A=\left(\begin{matrix} 1 & 1\\ 0 & 1
\end{matrix}\right).\label{mtph}\end{equation}
Here the latter asymptotics is uniform in $j=1,\dots,N_k$ for each individual $m$, as 
$k\to\infty$.  Then the products of the  matrices $M_{j;k}$ have  
the asymptotics  
\begin{equation} \wh M_k:=M_{N_k;k}\dots M_{1;k}
=\left(\begin{matrix} 1 & N_k\\ 0 & 1\end{matrix}\right) + o(\phi_{0k}^m) \ \text{  for every  } \ 
m\in\nn.\label{asmat}
\end{equation}
\end{proposition}
\begin{proof} Conjugation by the diagonal matrix $H_k:=\diag(1,\phi_{0k}^{-1})$ transforms 
the matrices $M_{j;k}$ and their product  respectively to the following matrices:  
$$\wt M_{j;k}=B_k+o(\phi_{0k}^m), 
 \ B_k=\left(\begin{matrix} 1 & \phi_{0k}\\ 0 & 1
\end{matrix}\right); \ \wt M_k:=\wt M_{N_k;k}\dots \wt M_{1;k}.$$

{\bf Claim 3.} {\it One has}
\begin{equation}\wt M_k=B_k^{N_k}+o(\phi_{0k}^m)=\left(\begin{matrix} 1 & N_k\phi_{0k}\\ 0 & 1
\end{matrix}\right)+o(\phi_{0k}^m).\label{renpro}\end{equation}

\begin{proof} Without loss of generality we can and will consider that $N_k\phi_{0k}\to C\in\rr_{\geq0}$, 
passing to a subsequence, since $N_k=O(\frac1{\phi_{0k}})$, by assumption. 
Let $\ut\subset\gl_2(\rr)$ denote the one-parametric 
subgroup of  unipotent upper triangular matrices. 
Consider the tangent vector 
$$V=\left(\begin{matrix} 0 & 1\\ 0 & 0
\end{matrix}\right)\in T_1\ut\subset T_1\gl_2(\rr).$$
Let us extend it to a left-invariant vector field on $\gl_2(\rr)$, which 
is tangent to the $\ut$-orbits under right multiplication action. Take a small 
transverse section $S\subset\gl_2(\rr)$ passing through the identity and consider 
the subset $\mathcal W\subset\gl_2(\rr)$ foliated by arcs of phase curves of the field $V$ 
starting in $S$ and parametrized by time segment $[0,2C]$. The subset $\mathcal W$ is a 
bordered domain (flowbox) diffeomorphic 
to the product $S\times[0,2C]$ via the diffeomorphism  sending a point  $y\in\mathcal W$ to the 
pair $(s(y),t(y))$ such that the orbit issued from the point $s(y)\in S$ arrives to $y$ in time 
$t(y)$. Fix an arbitrary $m\geq3$. 
In the new chart $(s,t)$ the multiplication by a matrix $\wt M_{j;k}=B_k+o(\phi_{0k}^m)$ 
from the right moves a point $(s,t)$ to the point $(s,t+\phi_{0k})$ up to a small correction of 
order $o(\phi_{0k}^m)$.  Therefore, the multiplication by $N_k\simeq\frac C{\phi_{0k}}$ 
similar matrices $\wt M_{j;k}$ with the $o(\phi_{0k}^m)$ in their asymptotics being uniform in $j$ 
 moves a point $(s,t)$ to a point $(s,t+N_k\phi_{0k})$ up to a 
correction of order $N_ko(\phi_{0k}^m)=o(\phi_{0k}^{m-1})$. This implies 
(\ref{renpro}) with $m$ replaced by $m-1$. Taking into account that $m$ can be choosen 
arbitrary, this proves (\ref{renpro}). 
\end{proof}

Conjugating formula (\ref{renpro})  by the matrix $H^{-1}_k$ and taking into account that 
$m\in\nn$ is arbitrary yields (\ref{asmat}). This proves Proposition \ref{propasm}.
\end{proof}

\begin{proof} {\bf of Proposition \ref{diffto0}.}  For $z\in\wh\Delta$ set 
$$St(z):=(\frac{\partial\wt\phi}{\partial\tau},\frac{\partial\wt\phi}{\partial\phi})(z).$$
The string of the first partial derivatives of the function 
$\wt\phi_N=\wt\phi\circ\wt F^N(x)$, $N=N(x)$, is equal to the product  
$$St(\tau_N,\phi_N)M(\tau_{N-1},\phi_{N-1})\dots M(\tau_0,\phi_0), \ (\tau_j,\phi_j)=\wt F^{j}(x), \ j=0,\dots,N,$$  
 \begin{equation}St(\tau_N,\phi_N)=\left(0,1\right)+o(\phi_0^m) \text{ for every } m\in\nn,\label{stringas}
 \end{equation}
  by $\phi$-flatness of the initial function $\wt\phi$ on $\wh\Delta$ and by the uniform asymptotics  $\phi_j=\phi_0(1+o(1))$, $j=1,\dots,N$  
  (Proposition  \ref{propit}, Statement c)). 
  
 Take arbitrary sequence of points 
 $x(k):=(\tau_{0k},\phi_{0k})$, $\tau_{0k}\in J$, $\phi_{0k}\to0$, as $k\to\infty$. Set  
  $$(\tau_{jk},\phi_{jk}):=\wt F^j(x(k)), \ N_k:=N(x(k)).$$
  The sequence of collections of Jacobian matrices  $M_{j+1;k}:=M(\tau_{jk},\phi_{jk})$, $j=0,\dots,N_k-1$, 
 satisfy the conditions of Proposition \ref{propasm}, by (\ref{asnx}) and (\ref{mjun}). Therefore, their product $\wh M_k$, which is the 
 Jacobian matrix of the differential $d\wt F^{N_k}(x(k))$, 
  has asymptotics (\ref{asmat}):
 \begin{equation}\wh M_k:= \text{ the Jacobian matrix of } d\wt F^{N_k}(x(k)) \ = \ 
\left(\begin{matrix} 1 & N_k\\ 0 & 1\end{matrix}\right) + o(\phi_{0k}^m).
\label{jacas}\end{equation}
  Thus, the matrix-string of the differential $d\wt\phi_{N_k}(\tau_{0k},\phi_{0k})$ 
 is the product  
 $$St(\tau_{Nk},\phi_{Nk})\wh M_k=\left(\left(0,1\right)+o(\phi_{0k}^m)\right)\left(\begin{matrix} 1 & N_k\\ 
 0 & 1
\end{matrix}\right)+o(\phi_{0k}^m)=\left(0,1\right)+o(\phi_{0k}^{m-1}),$$
since $N_k=O(\frac1{\phi_{0k}})$, see (\ref{asnx}). 
For $m=2$ we get that the differential  $d(\wt\phi_{N_k}(\tau,\phi)-\phi)$ taken at the point $x(k)$ tends to zero, as $k\to\infty$.  This proves Proposition \ref{diffto0}.
\end{proof}

Step 2: the higher derivatives. For a smooth function $f$ defined on a neighborhood of a point $x$ 
by $j^\ell_x(f)$ we will denote its $\ell$-jet at $x$. Below we 
 prove the following proposition.
 \begin{proposition} \label{difftol} In the conditions of Proposition \ref{diffto0} for every 
 $\ell\in\nn$ the $\ell$-jet at $x(k)$ of the difference $\wt\phi\circ\wt F^{N_k}-\phi$ 
 tends to zero, as $k\to\infty$.
 \end{proposition}
Proposition \ref{difftol} will imply $C^{\infty}$-smoothness and $\phi$-flatness 
of the extended function $\wt\phi$ 
at the points of the boundary interval $J\times\{0\}$. 

For every $\ell\in\nn$ and $x\in\rr^2$ 
 let $J^\ell_x$ denote the space of $\ell$-jets of functions at the point $x$.  
 The map $\wt F$ 
induces a transformation of functions, $g\mapsto g\circ\wt F$. This induces  
 linear operators in the jet spaces, $D_\ell\wt F(x):J^\ell_{\wt F(x)}\to J ^\ell_{x}$. 
We  identify the space of $\ell$-jets at  each point in $\rr^2$ with the $\ell$-jet space 
at the origin, which in its turn is identified with the 
space $\mcp_{\leq\ell}$ of polynomials in two variables of degrees no greater than $\ell$. 
Thus, we consider  the operator 
$D_\ell\wt F(x)$ as acting on the above space $\mcp_{\leq\ell}$. One has 
\begin{equation} D_\ell\wt F^N(x)=D_\ell\wt F(x)\dots D_\ell\wt F(F^{N-1}(x)).\label{dlf}
\end{equation}
Linear  changes of variables $(\tau,\phi)$ act on the space $\mcp_{\leq\ell}$ and 
induce an injective linear anti-representation $\rho:\gl_2(\rr)\to\gl(\mcp_{\leq\ell})$. Let 
$A$ denote the unipotent Jordan cell, see (\ref{mtph}). 
\begin{proposition} \label{propseq}
 For every sequence of points $x(k)=(\tau_{0k},\phi_{0k})\in W$ 
with $\phi_{0k}\to0$, as $k\to\infty$, set $N_k:=N(x(k))$, one has 
\begin{equation}D_\ell\wt F^{N_k}(x(k))=\rho(A^{N_k})+o(\phi_{0k}^m) \ \text{ for every } 
m\in\nn.\label{delwt}\end{equation}
\end{proposition}
\begin{proof} One has 
\begin{equation} D_\ell\wt F(\tau,\phi)=\rho(A)+\flat(\phi), \label{delfl}\end{equation}
by (\ref{dwtf}). Set $x_j(k)=(\tau_{jk},\phi_{jk})=\wt F^j(x(k))$, $j=0,\dots, N_k-1$. One has 
\begin{equation} D_\ell\wt F(x_j(k))=\rho(A)+o(\phi_{0k}^m) \text{ for every  } m\in\nn,
\label{delflj}\end{equation}
by (\ref{delfl}) and Proposition \ref{propit}, Statement c). We use (\ref{dlf}) and the 
following multidimensional version of Proposition \ref{propasm}.

\begin{proposition} \label{propasm2} 
Consider  arbitrary sequences of numbers $\phi_{0k}>0$,  $N_k\in\nn$, 
$\phi_{0k}\to0$, $N_k=O(\frac1{\phi_{0k}})$, as $k\to\infty$, and matrix collections 
$$\mathcal M_k=(M_{1;k},\dots,M_{N_k;k}), \ M_{j;k}\in\gl(\mcp_{\leq\ell}),$$
\begin{equation} M_{j;k}=\rho(A)+o(\phi_{0k}^m) \text{ for every } m\in\nn; \ 
A=\left(\begin{matrix} 1 & 1\\ 0 & 1
\end{matrix}\right).\label{mtph2}\end{equation}
Here the latter asymptotics is uniform in $j=1,\dots,N_k$ for each individual $m$, as 
$k\to\infty$.  Then the product of the  matrices $M_{j;k}$ has 
the asymptotics  
\begin{equation} \wh M_k:=M_{N_k;k}\dots M_{1;k}
=\rho(A^{N_k}) + o(\phi_{0k}^m) \text{ for every  } 
m\in\nn.\label{asmat2}
\end{equation}
\end{proposition} 
\begin{proof}
Conjugating the matrices $M_{j;k}$ by $\rho(H_k)$, $H_k:=\diag(1,\phi_{0k}^{-1})$, 
transforms them to matrices 
$$\wt M_{j;k}=\rho(B_k)+o(\phi_{0k}^{m'}), \ 
 \ B_k=\left(\begin{matrix} 1 & \phi_{0k}\\ 0 & 1
\end{matrix}\right), \ m'=m-\ell-1.$$

It suffices to show that the product of the matrices $\wt M_{j;k}$ has asymptotics 
$\rho(B_k^{N_k})+o(\phi_{0k}^m)$ for every  $m\in\nn$, as in  Claim 3. 
This is done by applying the arguments from the proof of Claim 3
to the left-invariant vector field on $\gl(\mcp_{\leq\ell})$ whose time $t$ flow map acts by 
 right multiplication by $\rho(A^t)$. 
\end{proof}

Formula (\ref{delwt}) is deduced from Proposition \ref{propasm2} and formulas 
(\ref{dlf}), (\ref{delflj}), as formula (\ref{jacas}).
\end{proof}

\begin{proof} {\bf of Proposition \ref{difftol}.} The polynomial representing 
the $\ell$-jet of the initial function 
$\wt\phi$ at a point $z\in\wh\Delta$ tends to the linear polynomial $P(\tau,\phi)=\phi$, 
as $z\to0$, so that its distance to $P(\tau,\phi)$ is $o(\phi^m)$ for every $m\in\nn$, by flatness of $\wt\phi$ on $\wh\Delta$. 
This together with Proposition \ref{propit}, Statement c) implies that the distance 
of its $\ell$-jet at the point $\wt F^{N_k}(x(k))$ to the polynomial $\phi$ 
is asymptotic to $o(\phi_{0k}^m)$. The image of the latter $\ell$-jet 
under the operator $D_\ell\wt F^{N_k}(x(k))$ is also $o(\phi_{0k}^m)$-close to $\phi$ 
for every $m\in\nn$. This follows from the previous statement, formula (\ref{delwt}), 
the fact that $\rho(A)$ fixes $\phi$ and the asymptotics $N_k=O(\phi_{0k}^{-1})$. 
Finally we get that the difference of 
the $\ell$-jet of the function $\phi$ at $x(k)$ and 
the  $\ell$-jet  $j^\ell_{x(k)}(\wt\phi\circ\wt F^{N_k})$ of the 
extended function tends to zero, as $k\to\infty$. 
Proposition \ref{difftol} is proved. 
\end{proof}

Lemma \ref{regflat} follows from Proposition \ref{difftol}. It implies 
Statement 1) of  Theorem \ref{thm33}.
\end{proof}
\subsection{Normal form. Proof of Statement 2) of Theorem \ref{thm33}}
Let $\wt\phi$ be a function from Statement 1) of Theorem \ref{thm33}. The 
 vector function $(\tau,\wt \phi)$ has non-degenerate Jacobian matrix at $J$. 
 Hence, shrinking $W$, we can and will consider that $(\tau,\wt\phi)$ are $C^{\infty}$-smooth 
 coordinates on $W\cup J$. In these coordinates 
 \begin{equation}\wt F: (\tau,\wt \phi)\mapsto(\tau+g(\tau,\wt\phi),\wt\phi), 
  \ \ g(\tau,\wt\phi)=\wt\phi+\flat(\wt\phi).\label{tph1}
\end{equation}
{\bf Claim.} {\it Shrinking $W$, one can achieve that there exists a $C^{\infty}$-smooth function 
$\wt\tau(\tau,\wt\phi)=\tau+\flat(\wt\phi)$ on $W\cup J$ such that $(\wt\tau,\wt\phi)$ are $C^{\infty}$-smooth 
coordinates on $W\cup J$ in which $\wt F$ acts as in (\ref{tauphiob}):}
\begin{equation}\wt F: (\wt \tau,\wt \phi)\mapsto(\wt\tau+\wt\phi,\wt\phi).\label{tauphiob2}
\end{equation}

\medskip

\begin{proof} Statement (\ref{tauphiob2}) is equivalent to the equation 
  \begin{equation}\wt\tau\circ F(x)=\wt\tau(x)+\wt\phi(x).\label{tauc}\end{equation}
 Shifting $\tau$ and shrinking $W$, we can and will consider that $(0,0)\in J$, 
 \begin{equation}g|_W>0, \ W\cap\{\wt\phi=\zeta\}=(\alpha(\zeta),\beta(\zeta))\times\{\zeta\} \ \text{ for all small } 
 \zeta>0,\label{gwf}\end{equation} 
 $$\alpha(\zeta)\to a, \ \beta(\zeta)\to b,\ \text{ as } \zeta\to0;  \ \ \wt F^{\pm1} \ \text{ are well-defined on } \ W\cup J.$$
  Fix a small $\chi\in(0,\frac12)$ and a $\eta>0$ (dependently on $\chi$) satisfying the statements of Proposition \ref{vi} 
  and such that the second statement (set equality) in (\ref{gwf}) holds 
  for every $\zeta<3\eta$.  Consider the sector $S_{\chi,\eta}$, the segment $K=\{0\}\times[0,\frac\eta2]$ and the fundamental domain $\Delta$ bounded by $K$, $F(K)$ and the (now horizontal) straightline segment connecting their 
  endpoints distinct from the origin. Set $\wh\Delta:=\overline\Delta\setminus\{(0,0)\}\subset S_{\chi,\eta}$.  First we define the function 
  $\wt\tau$ on the sector $S_{\chi,\eta}$ so that (\ref{tauc}) holds,  
whenever $x, \wt F(x)\in S_{\chi,\eta}$. Afterwards we extend $\wt\tau$ to all of $W$ by dynamics.
  
 Fix a $\sigma\in(0,\frac12(\frac12-\chi))$. Let $\rho_1$, $\rho_2$ be a partition of unity on the interval $(-\chi, 1+\chi)$ 
 subordinated to its covering by intervals $(-\chi,\frac12+\sigma)$, $(\frac12-\sigma,1+\chi)$, see (\ref{partun}). 
 Set $\nu:=\frac{\tau}{\wt\phi}$. For every $x\in S_{\chi,\eta}$ set 
 \begin{equation}\wt\tau(x):=\rho_1(\nu(x))\tau(x)+\rho_2(\nu(x))(\tau\circ\wt F^{-1}(x)+\wt\phi(x)).\label{wtxf}\end{equation}
The inclusion $x, \wt F(x)\in S_{\chi,\eta}$  implies (\ref{tauc}),  since then  
 $\rho_1(\nu(x))=1$ and $\rho_2(\nu\circ F(x))=1$, as in the proof of Proposition \ref{vi},  by  (\ref{wtxf}) 
 and $\wt F$-invariance of the function $\wt\phi$.  
 Recall that the height of the fundamental domain $\Delta$ is $\frac\eta2$. Let us now replace $W$ by 
 $W\cap\{\wt\phi<\frac\eta2\}$. Then for every $x\in W$ there exists a $N=N(x)\in\zz$ such that $\wt F^N(x)\in\wh\Delta$;  
 the latter $N$ is unique, unless $\wt F^N(x)\in\partial\wh\Delta$. This follows from (\ref{gwf}). Set 
 \begin{equation}\wt\tau(x):=\wt\tau(\wt F^N(x))-N\wt\phi(x).\label{defwt}\end{equation}
 The function $\wt\tau(x)$ is well-defined and $C^\infty$-smooth on all of $W$ and satisfies equation (\ref{tauc}) there. 
 Indeed, it suffices to check smoothness  on the  boundary $\partial\wh\Delta$ and on its images. If $x\in\partial\wh\Delta$,  then either $x, \wt F(x)\in\wh\Delta$, or $x, \wt F^{-1}(x)\in\wh\Delta$. In the first case one can take $N=0$ or $N=1$. 
 For both these values of $N$ the corresponding right-hand sides $\wt\tau(x)$ 
 and $\wt\tau\circ F(x)-\wt\phi(x)$ in (\ref{defwt}) coincide, since equation (\ref{tauc}) holds 
 on $S_{\chi,\eta}\supset\wh\Delta$. The second case is treated analogously. 
  For the same reason the function $\wt\tau$  on $S_{\chi,\eta}$ given by (\ref{wtxf}) 
 coincides with the corresponding expression (\ref{defwt}) (in which $N\in\{0,\pm1\}$). This implies 
 smoothness of the function (\ref{defwt}) on a neighborhood of the subset $\wh\Delta\subset W$. 
 One has $\wt\tau(\wt F(x))=\wt\tau(x)+\wt\phi(x)$, by (\ref{defwt}). This together with the above discussion implies 
 that $\wt\tau$ is $C^\infty$-smooth on $W$ and satisfies (\ref{tauc}) on all of $W$. The function 
 $\wt\tau$ extends to a $C^\infty$-smooth function on $W\cup J$, and the function $\wt\tau-\tau$ is $\wt\phi$-flat. This is proved as in 
 Subsection 2.4. Namely, fix an arbitrary compact segment  $[a',b']\subset(a,b)=J$.  The 
 differential of the map $H:(\tau,\wt\phi)\mapsto(\wt\tau,\wt\phi)$ tends   to the identity, and all its higher 
 derivatives tend to zero, as $\wt\phi\to 0$, uniformly 
 in $\tau\in[a',b']$. In particular, $\frac{\partial\wt\tau}{\partial\tau}(\tau,\wt\phi)\to1$. 
  Indeed, for every $x$, set $N=N(x)$, one has 
 $H(\tau,\wt\phi)=\wt F^N(\tau,\wt\phi)-(N\wt\phi,0)$ on a neighborhood of $x$. This together with formulas (\ref{jacas}) 
 and (\ref{delwt}) applied to the differential and higher jet 
 action of the iterates of the map $\wt F$ together imply the above convergence statement. The restriction of the 
 function $\wt\tau$ to $S_{\chi,\eta}$ extends continuously to the origin as $\wt\tau(0,0)=0$. Thus, 
 $\wt\tau(0,\wt\phi)\to0$, as $\wt\phi\to0$. This together with uniform convergence 
  $\frac{\partial\wt\tau}{\partial\tau}(\tau,\wt\phi)\to1$ in $\tau\in[a',b']$, as $\wt\phi\to0$, implies uniform convergence  
 $\wt\tau(\tau,\wt\phi)\to\tau$.  Together with the above higher derivative convergence, this implies $\wt\phi$-flatness 
 of the function $\wt\tau-\tau$ and proves the claim.
 \end{proof}
 
 The above claim immediately implies Statement 2) of Theorem \ref{thm33}. 
  
\subsection{Proof of existence in Theorem \ref{thm3}. Proof of Theorem \ref{addthm3}} 
Let $F$ be a $C^{\infty}$-lifted strongly billiard-like map. Let $(\tau,h)$ be the 
coordinates from Theorem \ref{thmm}. Set $\phi=\sqrt h$. Let $\wt F$ denote 
the map $F$ written in  the coordinates $(\tau,\phi)$, which is 
$C^{\infty}$-smooth and takes the form $(\tau,\phi)\mapsto(\tau+\phi+\flat(\phi), 
\phi+\flat(\phi))$ (Theorem \ref{thmm}). 
There exists a $\wt F$-invariant function $\wt\phi=\phi+\flat(\phi)$ 
(Theorem \ref{thm33}). 
The function $\wt h:=\wt\phi^2$ is $F$-invariant, 
$C^{\infty}$-smooth, and  $\wt h=h+\flat(h)$; hence $\frac{\partial\wt h}{\partial h}>0$ 
on $J$ and on some domain adjacent to $J$. The existence in Theorem \ref{thm3} is proved. 
Non-uniqueness of the function $\wt h$ will be proved in Subsection 2.9. 

Let us now prove Theorem \ref{addthm3}. Let us fix a function $\wt h$ 
constructed above. Let $\theta$ denote the time function of the Hamiltonian vector field 
with the Hamiltonian function $\wt h$, normalized to vanish on the vertical axis $\{\tau=0\}$. (We consider 
that $(0,0)\in J$, shifting the coordinate $\tau$.) 
The coordinates $(\theta,\wt h)$ are symplectic. In these coordinates 
$F(\theta,\wt h)=(\theta+\xi(\wt h), \wt h)$ for some function $\xi(\wt h)=\sqrt{\wt h}\psi(\wt h)$ in one variable, 
since $F$ preserves the symplectic area; $\psi$ is $C^\infty$-smooth and $\psi(0)>0$, as in Claim 1 in Subsection 2.1. 
Afterwards modifying 
the functions $\wt h$ and $\theta$, as at the end of Subsection 2.1, we get 
new coordinates $(\tau,\wt h)$ (with new $\tau$) in which $F$ takes the form 
(\ref{snform}). Theorem \ref{addthm3} is proved.

 \subsection{Foliation by caustics. Proof of existence in Theorems \ref{thm1},  \ref{thm2},  \ref{thm2closed}}
 
 First let us consider the case, when $\gamma$ is a strictly convex curve injectively parametrized by interval and bounding a domain in $\rr^2$ (conditions of Theorem \ref{thm1}).

 Let $W$ denote the domain in the space of oriented lines that consists of lines intersecting $\gamma$ twice and 
  satisfying condition b) from the beginning of Subsection 1.2. Let $\wh\gamma$ 
 denote the curve  given by the family of orienting tangent lines of $\gamma$. 
 The domain $W$ is adjacent to $\wh\gamma$. The billiard ball map is well-defined on $W\cup\wh\gamma$. 
  Each line $L$ close to a tangent line $\ell$ of $\gamma$ 
 carries a canonical orientation: the pullback of the orientation of the line $\ell$ under a projection $L\to\ell$ close to identity. 
 The billiard ball map acting on thus oriented lines close to tangent lines of $\gamma$  and intersecting $\gamma$ twice will be treated as 
a map acting on  non-oriented lines: we will just forget the orientation.   

Let us fix a natural length parameter $s$ on the curve $\gamma$ and identify each point in $\gamma$ with the corresponding length 
parameter value. 
Let us introduce the following tuples of coordinates on the domain $W$. For every line $L\in W$ let 
$s_1=s_1(L)$ and $s_2=s_2(L)$ denote the  length parameter values of its intersection points with $\gamma$. 
 Let $\phi_j$ denote the oriented angles between $L$ and the tangent lines to $\gamma$ at the points $s_j$. To each $L$ we put into correspondence 
 the pair $(s_1,\phi_1)$ where $s_j$ are numerated so that $s_1<s_2$. Set 
 $$y_1=1-\cos\phi_1,$$
 see (\ref{defy}). Any of the pairs $(s_1,\phi_1)$ or $(s_1,y_1)$ defines $L$ uniquely. Recall that $(s_1,y_1)$ are symplectic coordinates on $W$, see the discussion after 
 Remark \ref{1.9}. Let $V\subset\rr\times\rr_+$ denote the domain $W$ represented in the coordinates $(s_1,y_1)$. It is  adjacent to an interval 
 $J=(a,b)\times\{0\}$ representing $\wh\gamma$. Let $\wt V\subset\rr\times\rr_+$ denote the same domain represented in the coordinates $(s_1,\phi_1)$. 
 
  \begin{proposition}  In the coordinates $(s_1,y_1)$ the  billiard ball map is  a $C^{\infty}$-lifted strongly billiard-like map 
 $F$ defined on $V\cup J$. In the coordinates $(s_1,\phi_1)$ it is a $C^{\infty}$-smooth 
 diffeomorphism $\wt F$ defined on $\wt V\cup J$. 
  \end{proposition} 
 
 \begin{proof} The statements of the proposition follow from Proposition \ref{psmi} and Example \ref{exdel}.
\end{proof}

\begin{proposition} \label{pry1} Shrinking $V$ (without changing its boundary interval $J$), one can achieve that there 
exists a $C^{\infty}$-smooth $F$-invariant function $G(s_1,y_1)$ on $V\cup J$ such that 
$$G|_J\equiv0, \ \frac{\partial G}{\partial y_1}>0.$$ 
\end{proposition} 

\begin{proof} The proposition follows from Theorem \ref{thm3} (existence).
\end{proof}

From now on by $W$ we denote the domain of those lines that are represented by points of the 
(shrinked) domain $V$ from Proposition \ref{pry1}. 

The  level curves of the function $G$ are $F$-invariant and form a $C^\infty$-smooth  foliation. 
Lifting everything to the domain $W$ 
in the space of  lines we get a foliation by invariant curves under the billiard ball map. Each its leaf is a smooth family of 
lines. Its enveloping curve is a caustic of the billiard in $\gamma$. To prove that $\gamma$ and 
the caustics in question form a $C^{\infty}$-smooth foliation of a domain adjacent to $\gamma$, we use the following lemma. 

\begin{lemma} \label{lgw} The above function $G$ is $C^{\infty}$-smooth as a function on the domain with boundary $W\cup\wh\gamma$ in the 
space of lines. It has non-degenerate differential on $W\cup\wh\gamma$. 
 Thus, its level curves form a $C^{\infty}$-smooth foliation of $W\cup\wh\gamma$ with $\wh\gamma$ 
being a leaf.
\end{lemma}
\begin{remark} The function $s_1(L)$ is  smooth on $W$ but not on $W\cup\wh\gamma$: it is not $C^1$-smooth 
at points 
of the curve $\wh\gamma$. Therefore, a priori a function smooth in $(s_1,y_1)$ is not necessarily smooth on 
$W\cup\wh\gamma$. 
\end{remark}

\begin{proof}  {\bf of Lemma \ref{lgw}.} The function $G$ is $C^\infty$-smooth on $W$ and has non-degenerate 
differential there, by Proposition \ref{pry1}. Let us prove that this also holds at points of the boundary 
curve $\wh\gamma$. 
The function $G$ lifts to an $\wt F$-invariant function 
$$\wt G(s_1,\phi_1)=G(s_1,1-\cos\phi_1).$$
The map $(s_1,\phi_1)\to (s_1,s_2)$ is a diffeomorphism  defined on $\wt V\cup J$. The analogous statement 
holds for the diffeomorphism $(s_2,\phi_2)\mapsto(s_1,s_2)$. One has 
\begin{equation} \wt G(s_1,\phi_1)=\wt G(s_1,-\phi_1)=\wt G(s_2,\phi_2)=\wt G(s_2,-\phi_2),\label{sf12}\end{equation}
by invariance of the function $\wt G$ under sign change at the second coordinate and by its invariance under the 
billiard ball map represented by $\wt\delta_+:(s_1,\phi_1)\mapsto(s_2,-\phi_2)$. Given an unordered pair $(s_1,s_2)$, 
the tuples $(s_1,\phi_1)$ and $(s_2,\phi_2)$ are well-defined up to permutation. Therefore, in the 
coordinates $(s_1,s_2)$ on $\wt V$ (where $a<s_1\leq s_2<b$ by definition) the function $\wt G$ is $C^\infty$-smooth and extends 
$C^\infty$-smoothly to a neighborhood of the diagonal (identified with $J$) in $(a,b)\times(a,b)$ 
as a function invariant under coordinate permutation. This means that in the coordinates 
$$(\alpha,\beta):=(\frac{s_1+s_2}2, \frac{s_2-s_1}2)$$
(which are $C^{\infty}$-diffeomorphic coordinates on $\wt V\cup J$) the function $\wt G$ is invariant under sign change at $\beta$. 
Hence, $\wt G$ is a $C^\infty$-smooth function in $(\alpha,\beta^2)$, 
$$\wt G(s_1,\phi_1)=\wh G(\alpha,\psi), \ \ \psi:=\beta^2:$$ 
the function $\wh G$ is $C^\infty$-smooth on the domain in $\rr_\alpha\times(\rr_+)_\psi$ adjacent to $J$ and corresponding to $W$, and it is also smooth at points of the boundary $J$. 

\begin{proposition} \label{coorlines} 
The pair $(\alpha,\psi)$ forms $C^{\infty}$-smooth coordinates on the domain with boundary arc $W\cup\wh\gamma$ 
in the space of lines.
\end{proposition}
\begin{proof} Consider the map $S\La: (s_1,s_2)\to\{\text{lines}\}$ sending a pair of points $(s_1,s_2)$ of the curve 
$\gamma$ to the line through them. (For $s_1=s_2=s$, the image is the tangent line to $\gamma$ at $s$.) 
The map $S\La$  is  $C^{\infty}$-smooth  on $(a,b)\times(a,b)$. 
It is invariant under pertumation of the coordinates $s_1$, $s_2$, and its restriction to each connected component 
of the complement to  the diagonal is a diffeomorphism, by convexity.  
Equivalently, it is $C^{\infty}$-smooth in the coordinates $(\alpha,\beta)$ and invariant under sign change at $\beta$. Hence, it is smooth 
in $(\alpha,\psi)$. Its differential is non-degenerate at those points, where $s_1\neq s_2$, or equivalently, $\psi\neq0$. 
It remains to check that it has non-degenerate differential at the points of the line $\{\psi=0\}$. To do this, consider yet another 
tuple of coordinates $(\alpha^*,\psi^*)$ on $W\cup\wh\gamma$ defined as follows:
for every $L\in W\cup\wh\gamma$ 

- the point $\alpha^*=\alpha^*(L)\in(a,b)$ is the unique point in the curve $\gamma$ where the tangent line to $\gamma$ 
is parallel to $L$ (it exists by Rolle Theorem); 

- the number $\psi^*=\psi^*(L)$ is the  distance between the line $L$ and the above tangent line. 

\begin{proposition} \label{pra*}
The coordinates $(\alpha^*,\psi^*)$ are  $C^{\infty}$-smooth coordinates on $W\cup\wh\gamma$. 
\end{proposition}

Proposition \ref{pra*} follows from definition and strict convexity of  $\gamma$. 

Consider now $(\alpha^*,\psi^*)$ as functions of $(\alpha,\psi)$. One obviously has 
\begin{equation}\psi^*(\alpha,0)\equiv0, \ \alpha^*(\alpha,0)\equiv\alpha, \ \frac{\partial\alpha^*}{\partial\alpha}(\alpha,0)\equiv1, \ \frac{\partial\psi^*}{\partial\alpha}(\alpha,0)\equiv0.\label{albe}\end{equation}
As $(\alpha,\psi)\to(\alpha_0,0)$, one has 
\begin{equation}\psi^*\simeq\frac12\kappa(\alpha_0)\psi.\label{betas}\end{equation}
Here $\kappa$ is the curvature of the curve $\gamma$. Indeed, as $s_1,s_2\to \alpha_0$, the line $L$ through $s_1$ and $s_2$ is parallel to a line tangent  to $\gamma$ at a point 
$\alpha^*$ that is $o(s_1-s_2)$-close to $\alpha=\frac{s_1+s_2}2$.  
The distance between the two lines is asymptotic to $\frac12\kappa(\alpha)(\alpha^*-s_1)^2$, by \cite[formula (2.1)]{gpor}. 
This together with the equality  $\alpha-s_1=\beta\simeq\alpha^*-s_1$ implies (\ref{betas}), which in its turn implies that 
 $\frac{\partial\psi^*}{\partial\psi}(\alpha,0)>0$. Together with (\ref{albe}), this implies non-degeneracy of the Jacobian matrix 
 of the vector function $(\alpha^*(\alpha,\psi),\psi^*(\alpha,\psi))$ at the line $\{\psi=0\}$. This proves Proposition \ref{coorlines}.
 \end{proof}
 
 The function $\wh G(\alpha,\psi)=\wt G(s_1,\phi_1)$ is smooth in  $(\alpha,\psi)$, as was shown above. Hence, it is smooth on 
 $W\cup\wh\gamma$ (Proposition \ref{coorlines}). It remains to show that it has non-zero differential at each point $x\in\wh\gamma$; then shrinking $W$ we get that the differential is non-zero at each point in $W\cup\wh\gamma$. 
 Indeed, it is smooth in the coordinates $(s_1,y_1)$ (in which $y_1|_{\wh\gamma}\equiv0$), and one has 
 \begin{equation}\wt G(s_1,\phi_1)\simeq a(s_1)y_1(1+o(1)), \text{ as } y_1\to0, \ \ a(s)>0,\label{gy1}\end{equation}
   by Proposition \ref{pry1}. On the other hand,  as $s_1,s_2\to s$, one has $y_1,\phi_1\to0$ and 
   $$y_1=1-\cos\phi_1=\frac12\phi_1^2(1+o(1)),\ \ \psi=\left(\frac{s_2-s_1}2\right)^2, \ s_2-s_1\simeq2(\kappa(s))^{-1}\phi_1.$$
Hence, $y_1\simeq \frac12\kappa^2(s)\psi$. This together with (\ref{gy1}) implies that in the coordinates $(\alpha,\psi)$ 
one has $\frac{\partial \wh G}{\partial\psi}(\alpha,0)>0$. Together with the above discussion, this proves Lemma \ref{lgw}.
\end{proof}
 
\begin{proof} {\bf of existence in Theorem \ref{thm1}.} The function $G$  defined on 
the set $W\cup\wh\gamma$ in the space of lines is invariant under the billiard ball map. 
Therefore, its level curves are invariant families of lines. They form a $C^{\infty}$-smooth 
foliation of $W\cup\wh\gamma$, with $\wh\gamma$ being a leaf. Let us denote the latter foliation by $\mcf$. 
 The enveloping curves of the curve $\wh\gamma$ and of its leaves  are respectively the curve $\gamma$ and 
 caustics of the billiard on $\gamma$. 
  Let us show that they lie on its convex side and there exists a domain $U\subset\rr^2$ adjacent to $\gamma$ from the convex side such that 
  the latter caustics form a $C^{\infty}$-smooth 
 foliation of $U\cup\gamma$, with $\gamma$ being a leaf.
 
 Fix a projective duality sending lines to points, e.g., polar duality with respect to the unit circle centered at a point $O$ 
 in the convex domain bounded by $\gamma$. Let us shrink $W$ so that its points represent lines 
 that do not pass through $O$. Then the duality represents the subset $W\cup\wh\gamma$ 
 in the space of lines as a domain in the affine chart  $\rr^2\subset\rp^2$ with a boundary curve. 
 The latter domain and curve will be also denoted by $W$ and $\wh\gamma$ respectively.  The curve $\wh\gamma$ is 
 dual to $\gamma$. 
  
  \begin{proposition} The curve $\wh\gamma$ is strictly convex, and the domain $W$ lies on its concave side.
  \end{proposition}
  
  \begin{proof} The curve $\wh\gamma$ is strictly convex, being dual to the strictly convex curve $\gamma$. 
  Each point  $x\in W$ is dual to a line intersecting $\gamma$ twice. Therefore, there are two tangent lines to $\wh\gamma$ 
  through $x$. Hence, $x$ lies on the concave side from $\wh\gamma$.
  \end{proof}
  
  For every $x\in W$ let $\mcf_x\subset\rr^2\subset\rp^2$ denote the leaf through $x$ of the foliation $\mcf$ 
  (represented in the above dual chart), and let $L_x$ denote its projective tangent line at $x$.
  The enveloping curve of the family of lines represented by the curve $\mcf_x$ (treated now as a subset in the 
  space of lines) is its dual curve $\mcf_x^*$. It consists of 
  points $L_y^*$ dual to the lines $L_y$ for all $y\in\mcf_x$. Recall that the boundary curve $\wh\gamma$ is a 
  strictly convex leaf.
  
  \begin{proposition} \label{shrink} Shrinking the domain $W$ adjacent to $\wh\gamma$ one can achieve that the map 
  $x\mapsto L_x^*$ be a $C^{\infty}$-smooth diffeomorphism of the domain $W\cup\wh\gamma$ onto 
  a domain $U\subset\rr^2\subset\rp^2$ taken together with its boundary arc $\gamma$. The domain $U$ lies on the convex side from the curve $\gamma$. 
  \end{proposition}
  
  \begin{proof} The curve $\wh\gamma$ is strictly convex. No its tangent line passes through $O$, being dual to a 
 point of the curve $\gamma$ (which is a finite point).  Therefore, every compact arc in $\wh\gamma$ has a 
  neighborhood in $\rr^2$ whose intersection with each leaf of the foliation $\mcf$ is a strictly convex curve. 
  Thus, shrinking $W$ we can and will consider that each leaf $\mcl$ is strictly convex and no its tangent line  
 passes through $O$.  Hence, 
  each line $L$ tangent to $\mcl$ 
is disjoint from the leaves lying on the convex side from $\mcl$. Thus, $L$ is disjoint from  $\wh\gamma$ and 
$O\notin L$. 
  Let $U$ denote the set of points dual to lines tangent to leaves in $W$. 
  In the dual picture the latter statements mean that $U\subset\rr^2$ and for every   $A=L^*\in U$  there are no tangent 
  lines  to $\gamma$ passing  through $A$.  The set $U$ is path-connected, 
  disjoint from $\gamma$, and it accumulates to all of $\gamma$. Therefore, it approaches $\gamma$ from 
  the convex side, by the previous statement. Hence, it lies entirely on its convex side. 
  Let us now prove that shrinking $W$ one can achieve that the map $x\mapsto L_x^*$ be a diffeomorphism 
  $W\cup\wh\gamma\to U\cup\gamma$. 

  Fix a compact arc exhaustion 
  $$\wh\gamma_1\Subset\wh\gamma_2\Subset\dots=\wh\gamma.$$
  For every $k$ fix a flowbox $\Pi_k\subset W$ of the foliation $\mcf$ adjacent to 
  $\wh\gamma_k$ and lying in $W$ whose 
  leaves are strictly convex. We construct  the flowboxes $\Pi_k$ with decreasing heights,  
  which means 
  that for every $k$ each leaf of the flowbox $\Pi_{k+1}$ crosses $\Pi_k$. 
Now replace $W$ by the union $\cup_k\Pi_k$, 
  which will be now denoted by $W$. The leaves of the foliation on $W\cup\wh\gamma$ are strictly convex and 
  connected, by construction. We claim that the map $x\mapsto L_x$, and hence, $x\mapsto L_x^*$ is a 
  $C^\infty$-smooth  diffeomorphism. Indeed, it is a local diffeomorphism by strict convexity of leaves. 
  It remains to show that $L_x\neq L_y$ for every distinct $x,y\in W$. Indeed, fix an $x\in W$,  let $\mcl$ denote 
  the leaf of the foliation $\mcf$ through $x$. Set 
  $L=L_x$. Fix a $k$ such that $x\in\Pi_k$. 
  Every  leaf in the flowbox $\Pi_k$ that does not lie in its leaf through $x$  
  either intersects $L$ transversally, or is disjoint from $L$, 
  by convexity. Then the latter statement also holds for every other flowbox $\Pi_\ell$,  by construction 
  and convexity.  This implies that $L$ 
  can be tangent to no other leaf in $W$. It cannot be tangent to the same leaf  $\mcl$ at another point $y\neq x$, 
  by convexity and the above statement.  This proves diffeomorphicity of the map 
  $x\mapsto L_x^*$. 
  \end{proof}
   
The above $C^{\infty}$-smooth diffeomorphism $x\mapsto L_x^*$ sends $W$ onto a domain $U\subset\rr^2$ adjacent 
to $\gamma$. It sends leaves of the foliation 
  $\mcf$ to the corresponding caustics of the billiard on $\gamma$. 
  Hence, the caustics together with the curve $\gamma$
   form a $C^{\infty}$-smooth foliation of $U\cup\gamma$. Constructing the above flowboxes 
   $\Pi_k$ narrow enough in the transversal direction 
   (step by step), we can achieve that for every $x\in\gamma$ and every leaf $\mcl$ 
   of the foliation $\mcf$ there are at most two tangent lines through $x$ 
     to the caustic $\mcl^*$. Indeed, each leaf $\mcl$ of the foliation $\mcf$ is a leaf of some flowbox $\Pi_k$. 
     Its dual caustic $\mcl^*$ will satisfy the above tangent line statement, if the total angle 
     increment of its tangent vector is no greater than $\pi$. 
     The latter angle increment statement holds for the curve $\gamma$. 
     Hence, it remains valid for the caustics dual to the leaves 
     of the flowbox   $\Pi_k$, if $\Pi_k$ is chosen   narrow enough. 
   The existence statement of Theorem \ref{thm1} is proved. 
 \end{proof}
 
 The proof of the existence in Theorem \ref{thm2}  repeats the above proof of the existence in 
 Theorem \ref{thm1} with obvious changes. The existence statement of Theorem \ref{thm2closed} follows from 
 that of Theorem \ref{thm2}.

 \subsection{Space of foliations. Proofs of Theorems \ref{unget},  \ref{tmodgerm}, \ref{tmog} and  Proposition \ref{pmodgerm}}
 
 \begin{proof} {\bf of Theorem \ref{unget}.} It suffices to prove the statement Theorem \ref{unget} for foliations from 
 Theorems \ref{thm3}, \ref{thm33}, since the foliations in Theorems \ref{thm1}, \ref{thm2}, \ref{thm2closed} are obtained from  foliations in Theorem \ref{thm3} by duality, see the above subsection. 
 
 Case of Theorem \ref{thm3}. Consider a $C^{\infty}$-lifted strongly billiard-like map. We already know that in appropriate coordinates it takes the form (\ref{snform}): 
 \begin{equation}F(\tau, h)=(\tau+\sqrt{h},h).\label{snform2}\end{equation}
 The function $h$  is $F$-invariant, and so are its level lines. 
 
 Suppose the contrary: there exists another $C^{\infty}$-smooth $F$-invariant function $G(\tau,h)$, $G(\tau,0)\equiv0$, 
 without critical points on $J=\{ h=0\}$ and such that there exists an $x\in J$ where  the foliations $G=const$ and $h=const$ have 
 different $k$-jets for some $k$. Without loss of generality we consider that $x=(0,0)$, shifting the coordinate $\tau$. 
 Then the asymptotic Taylor series of the function $G$ at $x$ contains 
 at least one monomial $a_{mn}\tau^mh^n$ with $m\geq1$ and a non-zero coefficient $a_{mn}$. Set 
 $$\mcn:=\{(m,n) \ | \ a_{mn}\neq0, \ m\geq1\}, \ d:=\min\{ m+2n \ | \ (m,n)\in\mcn\}.$$
 Consider the lower $(1,2)$-quasihomogeneous part: 
 $$G_d(\tau,h):=\sum_{(m,n)\in\mcn, \ m+2n=d}a_{mn}\tau^mh^n+a_{0,\frac d2}h^d.$$
 One has $G\circ F(\tau,h)-G(\tau,h)\equiv0$. On the other hand, 
$$G\circ F(\tau,h)-G(\tau,h)=G(\tau+\sqrt h,h)-G(\tau,h)$$
 \begin{equation}=G_d(\tau+\sqrt h,h)-G_d(\tau,h)+\text{ higher terms}.
 \label{g-g}\end{equation}
 Here "higher terms" means "a  function that admits an asymptotic Taylor series in $(\tau,\sqrt h)$ at $(0,0)$ 
 that  contains only terms $a_{mn}\tau^\alpha h^\beta$  of quasihomogeneous degrees $\alpha+2\beta>d$". 
 Let $m_0$ denote the higher degree of $\tau$ in a monomial entering $G_d$. 
 The difference $G_d(\tau+\sqrt h,h)-G_d(\tau,h)$ is quasihomogeneous of degree $d$. It contains the monomial $\tau^{m_0-1}h^\ell$, $m_0-1+2\ell=d$, 
 with non-zero coefficient, by construction; here a priori $\ell$ may be non-integer. 
 This monomial will not cancel out with other monomials in the asymptotic 
 Taylor series of the difference $G\circ F(\tau,h)- G(\tau,h)$, by 
 construction. Therefore, the latter difference cannot be identically equal to zero. The contradiction thus obtained 
 proves the statement of  Theorem \ref{unget} in the conditions of Theorem \ref{thm3}.
 
 Case of Theorem \ref{thm33}. In this case we know that the map in question is conjugated to 
 $(\tau,\phi)\mapsto(\tau+\phi,\phi)$. The statement of Theorem \ref{unget} for the latter map is 
 proved by the above arguments for lower homogeneous (i.e., $(1,1)$-quasihomogeneous) terms of the 
 Taylor series of the function $G$. 
 \end{proof}
 
 \begin{proof} {\bf of Proposition \ref{pmodgerm}.} Let $g(\tau,h)$ be a $C^\infty$-smooth function  invariant under the map 
 $F:(\tau,h)\mapsto(\tau+\sqrt h,h)$ that has type (\ref{gnorma}): 
 \begin{equation} g(\tau,h)=h+\flat(h), \  \ g(0,h)\equiv h.\label{gnorman}\end{equation}
Invariance is equivalent to the equality $g(\tau+\sqrt h,h)=g(\tau,h)$. This together with (\ref{gnorman}) implies that  
\begin{equation}\text{the function } \ \psi(s,h):=g(s\sqrt h,h)-h \  \text{ is 1-periodic in } s, \ \psi(0,h)=0.\label{psiper}\end{equation}
Moreover, the function $\psi(s,h)$ is  $C^{\infty}$-smooth and $h$-flat   on  a cylinder $S^1\times[0,\delta)$, $S^1=\rr_{s}\slash\zz$, 
for a small $\delta>0$. This follows from smoothness and $h$-flatness of the function $g(\tau,h)-h$. Conversely, consider 
an $h$-flat function $\psi(s,h)$ that is 1-periodic in $s$ and such that $\psi(0,h)=0$. Then the function 
$$g(\tau,h):=\psi(\frac\tau{\sqrt h},h)+h$$ 
is $C^{\infty}$-smooth,  $F$-invariant and its difference with $h$ is $h$-flat, by construction. Statement 1) of Proposition \ref{pmodgerm} is proved. Its Statement 2) can be reduced  to Statement 1) and also can be proved analogously. 
\end{proof} 
  
 Theorem \ref{tmodgerm}  follows immediately from Proposition \ref{pmodgerm}, and in its turn, it immediately implies Theorem \ref{tmog}.

 \subsection{Proof of Proposition \ref{distgerm} and non-uniqueness in main theorems} 
 
  \begin{proof} {\bf of Proposition \ref{distgerm}.} 
Let us prove the statement of Proposition \ref{distgerm} for a map 
$\wt F$ of type (\ref{tauphi}). We prove it for line fields: for other objects the 
proof is analogous. Without loss of generality we can and will consider that the map $\wt F$  takes the 
form $(\tau,\phi)\mapsto(\tau+\phi,\phi)$: see Statement 2) of Theorem \ref{thm33}. 
Let $G_1$ and $G_2$ be two $\wt F$-invariant line fields on $W$
with distinct germs at $J$. 
This means that there exists a sequence of points $x(k)=(\tau(k),\phi(k))$ with 
$\phi(k)\to0$ and $\tau(k)$ 
lying in a compact subset in $J$ such that the lines $G_1(x(k)),G_2(x(k))\subset T_{x(k)}\rr^2$ 
are distinct. Taking a subsequence, we can and will 
consider that $x(k)\to x=(\tau_0,0)$, as $k\to\infty$. The two-sided orbit 
of a point $x(k)$ with big $k$ consists of points with $\phi$-coordinate 
 $\phi(k)$  whose $\tau$-coordinates form an arithmetic progression with step $\phi(k)$ converging to zero.
At each point of the orbit the lines of the fields $G_1$ 
and $G_2$ are distinct, since this holds at $x(k)$ and by $\wt F$-invariance. 
Therefore, passing to limit, as $k\to\infty$, we get that for every point $z\in J$ 
there exist points $z'$ arbitrarily close to $z$ with $G_1(z')\neq G_2(z')$.
 Hence, the germs at $z$ of the line fields $G_1$ and $G_2$ are distinct. 
 The first statement of Proposition \ref{distgerm}, for a map $\wt F$ of type (\ref{tauphi}),  
 is proved.   Its second  statement, for a $C^\infty$-lifted strongly billiard-like map 
 $F:V\cup J\to  F(V\cup J)\subset\rr^2$
  follows from its first statement and the fact that 
 $F$  is conjugated to the map $(\tau,\phi)\mapsto(\tau+\phi,\phi)$  by a homeomorphism  that is smooth on the complement to the boundary interval $J$. The latter conjugating homeomorphism is the 
 composition of a diffeomorphism conjugating $F$ to the map $(\tau,h)\mapsto(\tau+\sqrt h,h)$ 
 (Theorem  \ref{addthm3}) and  the map $(\tau,h)\mapsto(\tau,\phi)$, $\phi=\sqrt h$. 
 Proposition \ref{distgerm} is proved.
 \end{proof}
 
 \begin{proof} {\bf of non-uniqueness  in Theorems \ref{thm33}, \ref{thm3}, \ref{thm1}, 
 \ref{thm2}, \ref{thm2closed}} 
 
Existence of continuum  of distinct {\it germs} of foliations satisfying the statements of any of the above-mentioned 
theorems follows immediately from Proposition \ref{distgerm} and Theorem \ref{tmog}, which states 
that there are as many distinct boundary germs, as many flat functions on the cylinder $S^1\times\rr_{\geq0}$ 
 with distinct germs at $S^1\times\{0\}$. It remains to show that 
there exists a domain adjacent to the boundary interval (or the curve $\gamma$) which admits an infinite-dimensional 
family  of corresponding foliations with distinct germs. 

Case of Theorem \ref{thm33}. Fix coordinates $(\tau,\phi)$ in which $\wt F(\tau,\phi)=(\tau+\phi,\phi)$. 
Recall that the coordinates $(\tau,\phi)$ are defined on $W\cup J$, where $W\subset\rr\times\rr_+$ is a domain 
adjacent to the interval $J=(a,b)\times\{0\}$. 
Fix a $C^\infty$-smooth $h$-flat finction $\psi$ on the cylinder $S^1\times\rr_{\geq0}$, $S^1=\rr\slash\zz$, with 
non-trivial germ at $S^1\times\{0\}$: 
\begin{equation}\psi(s,h)=\psi(s+1,h)=\flat(h), \ \ \psi(0,h)=0, \ \ |\psi|<\frac18.\label{psipsi}\end{equation}
For every $\var>0$ the function 
\begin{equation} g_\var(\tau,\phi):=\phi+\var\chi(\tau,\phi), \ \ \chi(\tau,\phi):=\psi(\frac{\tau}\phi,\phi),
\label{gepsi}\end{equation} 
is $\wt F$-invariant, $C^\infty$-smooth and well-defined 
on $W\cup J$. Let us show that shrinking the domain $W$ one can achieve that the foliation $g_\var=const$ 
is regular, that is $g_\var$ has no critical points on $W$, whenever $\var$ is small enough.

{\bf Claim 1.} {\it Replacing $W$ by a smaller domain adjacent to $J$, one can achieve that  
each    partial derivative of the function $\chi$ (of any order) 
be bounded on $W$. For any given $m$ and every $\delta>0$ 
shrinking $W$ 
(dependently on $m$ and $\delta$) one can achieve 
that all its order $m$ partial derivatives 
have moduli less than $\delta$.}

\begin{proof} The modulus of each partial derivative of order at most $m$ admits an upper bound by a quantity 
\begin{equation} |\frac{\partial^m\chi(\tau,\phi)}{\partial \tau^\ell\partial\phi^{m-\ell}}|\leq c_m(1+|\tau|^m)(1+\phi^{-2(m+1)})\sum_{\ell,r=1}^m|\psi_{\ell r}|,\label{ubou}\end{equation} 
$$c_m=const>0, \ \ \psi_{\ell r}(\tau,\phi)=\frac{\partial^{\ell+r}\psi}{\partial s^\ell\partial\phi^r}(\frac{\tau}\phi,\phi)=
o(\phi^k),   \ \text{ for every }\  k\in\nn.$$ 
Here the latter $o(\phi^k)$ is uniform in $\tau$, as $\phi\to0$. 
Estimate (\ref{ubou}) follows from 1-periodicity and flatness of the function $\psi(s,h)$ and chain rule for calculating derivatives. Let us now replace the domain $W$ by a smaller domain adjacent to $J$ on which 
the right-hand side in (\ref{ubou}) is bounded for each $m$ and is less than $\delta$ for a given $m$. 
First let us replace $W$ by the connected component adjacent to $J$ of its intersection with the strip $\{ a<\tau<b\}$. 
 In the case, when $(a,b)$ is a finite interval, 
the right-hand side in (\ref{ubou}) is uniformly bounded on $W$ and tends to zero uniformly in $\tau\in(a,b)$, 
as $\phi\to0$: the asymptotics $\psi_{\ell r}(\tau,\phi)=o(\phi^{3m+3})$ kills polynomial growth of the function $\phi^{-2(m+1)}$. 
Therefore, shrinking $W$ one can achieve that for  given $m$ and $\delta$, 
the right-hand side in (\ref{ubou}) be less than $\delta$ on $W$. 

In the case, when some (or both) of the boundary points $a$ or $b$  is infinity, take an exhaustion 
of the interval $(a,b)$ by segments $[a_k,b_k]$. By the above argument, we can 
take a rectangle $\Pi_k=(a_k,b_k)\times(0,d_k)\subset W$ on which for all $m$ the right-hand sides 
in (\ref{ubou}) be bounded, and for  some given $m$ the same right-hand side be less than a given $\delta$. 
  Replacing $W$ by $\cup_k\Pi_k$, 
we achieve that the two latter inequalities hold on $W\cup J$. 
\end{proof}

Let $W$ satisfy the statements of the above claim so that each first partial derivative of the function $\chi$ 
has modulus less than $\frac12$. Then for every $\var\in[0,1]$ 
the  foliation $g_\var=const$  is regular on $W\cup J$, since for those $\var$ one has 
$\frac{\partial g_\var}{\partial\phi}=1+\var\frac{\partial\chi}{\partial\phi}>\frac12$ on $W\cup J$. 
All its leaves are $\wt F$-invariant. For distinct values of the 
parameter $\var$ the germs of the corresponding foliations are distinct at each point of the interval $J$, 
 by Theorem \ref{tmog} and Proposition \ref{distgerm}. This yields a one-dimensional family of foliations 
 from Theorem \ref{thm33} with pairwise distinct germs at each point in $J$. 
 
 Now let us apply the above argument  with the expression $\var\chi$ in (\ref{gepsi})  being replaced by an arbitrary linear combination 
 \begin{equation}\wt\chi_\var(\tau,\phi)=\sum_{k=1}^N\frac{\var_k}{k!4^k}\psi^k(\frac{\tau}{\phi},\phi), \ \  
 \var=(\var_1,\dots,\var_N)\in[0,1]^N.\label{manye}\end{equation}
 Recall that $|\psi|<\frac18$. This inequality together with the above assumption that the first partial derivatives 
 of the function $\chi(\tau,\phi)=\psi(\frac{\tau}{\phi},\phi)$ have moduli less than $\frac12$ on $W$ 
 imply that for every $\var$ as in (\ref{manye}) the module of each first 
  partial derivative of the function 
 $\wt\chi_\var$ is less than $\frac12$ on $W\cup J$.  This implies  that the
   foliation by level curves of the function 
 $g_\var(\tau,\phi)=\phi+\wt\chi_\var$ is a $C^\infty$-smooth foliation on $W\cup J$.  
 We get a $N$-dimensional family of foliations on $W\cup J$ depending on $(\var_1,\dots,\var_N)\in[0,1]^N$ 
 with pairwise distinct germs at $J$, and hence, at each point of the curve $J$ (Proposition \ref{distgerm}).  
The non-uniqueness statement of Theorem \ref{thm33} is proved.

Case of Theorem \ref{thm3}. Its non-uniqueness statement follows from that of Theorem \ref{thm33} 
and also from the above arguments.

Case of Theorem \ref{thm1}. Let us consider the billiard ball map acting on lines as a $C^\infty$-lifted 
strongly billiard-like map $F$.  
Let us introduce new (symplectic) coordinates $(\tau,h)$ in which $F(\tau,h)=(\tau+\sqrt h,h)$, see (\ref{snform}). 
The map $F$ is defined on $W\cup J$, where $J=(a,b)\times\{0\}$ parametrizes the family of lines 
tangent to $\gamma$ and $W\subset\rr\times\rr_+$ is a domain adjacent to $J$. 
Representing lines as points in  $\rp^2$ via a 
projective duality $\rp^{2*}\to\rp^2$ transforms $J$ to a strictly 
convex curve  $\gamma^*\subset\rp^2$ dual to $\gamma$, and 
$W$ to a domain adjacent to $\gamma^*$ from the concave side. See Subsection 2.7. 
We can and will consider that 
$\gamma^*$ and $W$ lie  in an affine chart $\rr^2$, as in the proof of the existence in Theorem \ref{thm1} in Subsection 2.7. In what follows we identify $J$ with $\gamma^*$. 
Consider the foliation $h=const$ by $F$-invariant curves. Let us  construct a family of foliations 
using a $C^\infty$-smooth $h$-flat function $\psi(s,h)$ on $S^1\times\rr_{\geq0}$ with non-trivial germ at $S^1\times\{0\}$, 
as in (\ref{psipsi}). Namely, set 
$$g_\var(\tau,h):=h+\var\chi(\tau,h), \ \ \chi(\tau,h):=\psi(\frac{\tau}{\sqrt h},h).$$
The functions $g_\var$ are $F$-invariant. The germs of any two foliations $g_{\var_1}=const$, 
$g_{\var_2}=const$, $\var_1\neq\var_2$, are distinct at 
each point in $J$, by Theorem \ref{tmog} and Proposition \ref{distgerm}. It remains to prove their regularity 
and regularity of  the dual foliations by caustics on one and the same domain. To do this, we use the following claim. 

{\bf Claim 2.} {\it Shrinking the domain $W$ adjacent to $J=\gamma^*$ one can achieve that 
for every $\var\in[0,1]$ the foliation $g_\var=const$  is regular on $W\cup J$,  its 
leaves are strictly convex curves, as is $\gamma^*$, and the map  $\La_\var:x\to L_{x,\var}$ 
sending a point $x\in W$ to the projective line $L_{x,\var}$ 
tangent to the level curve $\{ g_\var=g_\var(x)\}$ at $x$ is a diffeomorphism on $W$.}  

\begin{proof} Consider the function $h$ and the  above function $\chi$ as functions on 
$W\cup\gamma^*$ as on a domain in $\rr^2\subset\rp^2$. The curve $\gamma^*=\{ h=0\}$ is strictly convex. 
Hence, shrinking $W$ we can and will consider that each level curve $\{ h=const\}\cap W$ is strictly convex. 
Consider the rectangles $\Pi_k\subset W$ from the proof of the above Claim 1 (in the coordinates $(\tau,h)$) 
 with decreasing heights. They are 
represented as curvilinear 
quadrilaterals in $\rr^2\subset\rp^2$. Choosing them with heights small enough,  we can achieve that 
$||\nabla\chi||<\frac12||\nabla h||$  on $\Pi_k$. Let us now replace $W$ by $\cup_k\Pi_k$. 
Then $\nabla g_\var\neq0$ on $W$, and hence, the foliation $g_\var=const$ is regular for all $\var\in[0,1]$. 
Choosing $\Pi_k$ with heights small enough  (step by step) 
 one can also achieve that each level curve $\{ g_\var=const\}\cap\Pi_k$ be 
strictly convex for every $\var\in[0,1]$, by strict convexity of the boundary curve $\gamma^*$ and 
 $h$-flatness of the function $\psi$. In more detail, let $(x,y)$ be coordinates on the ambient affine chart 
 $\rr^2$. Strict convexity of level curves $\{g_\var=const\}$ is equivalent to non-vanishing of the Hessian\footnote{The  Hessian $H(g)$ of a function $g$ was introduced by S.Tabachnikov in his paper \cite{tab08}, where he 
 used it to study his conjecture stating that every polynomially integrable outer billiard is an ellipse 
 (later this conjecture was solved in \cite{gs}). The Hessian was  also used  by M.Bialy, A.E.Mironov  and later the author   in the solution of Bolotin's polynomial version of the Birkhoff 
 Conjecture, which is the result of papers  \cite{bm, bm2, gl2}.} $H(g_\var)$:  
 $$H(g_\var)\neq0, \ \ H(g):=\frac{\partial^2g}{\partial x^2}\left(\frac{\partial g}{\partial y}\right)^2+\frac{\partial^2g}{\partial y^2}\left(\frac{\partial g}{\partial x}\right)^2-2\frac{\partial^2g}{\partial x\partial y}\left(\frac{\partial g}{\partial x}\right)\left(\frac{\partial g}{\partial y}\right).$$
 The Hessian $H(g_\var)$ is the sum of the Hessian $H(h)$ (which is non-zero on $W\cup\gamma^*$, 
 since the curves $\{h=const\}$ are strictly convex) and a finite sum of products; each product  contains 
  $\var$, at least one derivative of the function $\chi$ and at most two derivatives of the function $h$; each derivative 
  is of order at most two. Choosing the rectangles $\Pi_k$ with heights small enough, we can achieve that 
  the module of the  latter sum of products be no greater than $\frac12\var|H(h)|$ for $\var\in[0,1]$.  This follows from
  convexity of the curve $\gamma^*$ and $h$-flatness of 
   the function $\psi$: shrinking $W$, one can achieve that all the first and second derivatives of the function 
   $\chi$ have moduli bounded by arbitrarily small  $\delta$ (Claim 1). 
  Then $H(g_\var)\neq0$ on $W$, hence, the curves $\{ g_\var=const\}\cap W$ are strictly convex. 
  
  Now for every $k$ we choose smaller rectangles $\wt\Pi_k\subset\Pi_k$ 
  with decreasing heights and with the lateral (i.e., vertical) sides lying in the lateral sides of the bigger rectangles $\Pi_k$ 
  that satisfy the following additional statement.
  For every $\var\in[0,1]$ let $\Pi_{k,\var}$ denote the minimal flowbox for the foliation $g_\var=const$ 
  with lateral (i.e., transversal) sides lying in the lateral sides of $\Pi_k$ that contains $\wt\Pi_k$. This is the union of arcs of leaves 
  that go from one lateral side of $\Pi_k$ to the other one and cross $\wt\Pi_k$.  For every $k$ we can and will subsequently choose $\wt\Pi_k$ with heights small enough (i.e., narrow enough in the transversal direction) so that 
  for every $\var\in[0,1]$ the flowbox $\Pi_{k,\var}$ lies in $\Pi_k$, and the heights of the flowboxes $\Pi_{k,\var}$ 
  be decreasing in $k$: more precisely, for every $k$ each local leaf in $\Pi_{k+1,\var}$ crosses $\Pi_{k,\var}$, 
 as in the proof of Proposition \ref{shrink}.  Then 
  the map $\La_\var:x\mapsto L_{x,\var}$ 
   is a diffeomorphism on $W_\var:=\cup_k\Pi_{k,\var}$ for every $\var\in[0,1]$, as at the end of the proof of 
   Proposition \ref{shrink}.  Hence, it is a diffeomorphism on 
  \begin{equation} \wt W:=\cup_k\wt\Pi_k.\label{wtww}\end{equation}
    The claim is proved.
  \end{proof}
   
   {\bf Claim 3.} {\it Consider the foliation by caustics of the billiard on $\gamma$ that is dual to the foliation 
   $g_\var=const$. There exists a domain $U\subset\rr^2$ adjacent to $\gamma$ from the convex side where the 
   above foliation by caustics is $C^{\infty}$-smooth (and also smooth 
   at the points of the curve $\gamma$) for every $\var\in[0,1]$. Moreover, shrinking $U$  
   one can achieve that for every $x\in\gamma$ 
   and every $\var\in[0,1]$ there are at most two tangent lines through $x$ to any given leaf of the corresponding 
   foliation by caustics on $U$.}
   
   \begin{proof} Let  $\wt W$ be the domain (\ref{wtww}) constructed above. For every $\var\in[0,1]$ 
   the map $\La_\var^*:x\mapsto L_{x,\var}^*$ sending $x$ to the point dual to the corresponding line $L_{x,\var}$ 
   is a diffeomorphism, since so is $\La_\var$. It sends the domain $\wt W$ foliated by level curves of the 
   function $g_\var$ onto a domain $U_\var$ 
   adjacent to $\gamma$ and foliated by their dual curves: caustics of the billiard on 
   $\gamma$. They form a $C^\infty$-smooth foliation on $U_\var\cup\gamma$. 
   For the proof of the first statement of Claim 3 it remains to show that there exists a domain $U$ adjacent to $\gamma$ that lies in the intersection 
    $\cap_\var U_\var$ (and hence, for each $\var$ it is smoothly foliated by the corresponding caustics). To do this, we 
 construct the above $\wt W$ and a smaller domain $W'\subset\wt W$ adjacent to $\gamma^*$ 
   so that the following statement holds:
   
   (*) for every $p\in W'$ and every $\var\in[0,1]$ there exists a $q=q(p,\var)\in\wt W$ such that 
   the projective line $L_{p,0}$ tangent to the curve $\{ h=h(p)\}$ at $p$ is tangent to the leaf of the foliation $g_\var=const$ at $q$. 
   
   Statement (*) implies that the 
   image $U=\La_0(W')$ is contained in all the domains $U_\var$ and regularly foliated by 
   caustics dual to level curves of the function $g_\var$ for every $\var\in[0,1]$. 
   
   Take the rectangles  $\Pi_k$ and $\wt\Pi_k$  from the proof of Claim 2. Let us call their sections 
   $h=const$ horizontal and 
   transversal sections $\tau=const$ vertical.  For every $k$ fix two vertical sections $\ell_{1,k}$ and $\ell_{2,k}$ 
   crossing the interior $Int(\Pi_k)$ that lie in the $\frac1{2^k}$-neighborhoods of the corresponding 
   lateral sides of the 
   rectangle $\Pi_k$. We can and will choose a rectangle $\Pi_k'\subset\wt\Pi_k$ with lateral sides 
  lying on $\ell_{1,k}$ and $\ell_{2,k}$ and height small enough so that  for every $\var\in[0,1]$ 
  and every $p\in\Pi_k'$ there exists a $q=q(p,\var)\in\wt\Pi_k$ satisfying statement (*).  This is possible 
  by flatness of the function $\psi$ and strict convexity of the curve $\gamma^*$. Then 
  statement (*) holds for the domain $W'=\cup_k\Pi_k'\subset W$. This together with the above discussion proves 
  the first statement of Claim 3. One can achieve that its second statement (on tangent lines) hold as well by choosing 
  the above rectangles $\Pi_k'$  with heigth small enough, as in the proof of the existence in Theorem \ref{thm1} at the 
  end of Subsection 2.7.
  \end{proof}
  
  Claim 3  implies non-uniqueness statement of Theorem \ref{thm1}, with one-dimensional family of 
  foliations with distinct germs. Modifying the above arguments  
   as in the proof of non-uniqueness statement 
  of Theorem \ref{thm33} (see formula (\ref{manye}) and the discussion after it) we get $N$-dimensional 
  family of foliations with distinct germs for every $N\in\nn$.   Non-uniqueness statements 
  of Theorems  \ref{thm2} and \ref{thm2closed} are proved analogously.
\end{proof}
 
\subsection{Conjugacy of billiard maps and Lazutkin length. Proof of Theorems \ref{thconj1}, \ref{thconj2}, 
\ref{thconj3}, \ref{thconj4}, Propositions \ref{proform}, \ref{proform2} and Lemma \ref{lazconv}}

\begin{proof} {\bf of Proposition \ref{proform}.} The map 
$$(s,y)\mapsto(X,Y):=(t_L(s),w^{\frac23}(s)y), \ t_L(s)=\int_{s_0}^sw^{-\frac23}(u)du$$
is symplectic and conjugates $F$ to a $C^\infty$-lifted strongly billiard-like map  of the type 
\begin{equation} \Phi:(X,Y)\mapsto(X+\sqrt Y+O(Y), Y+o(Y^{\frac32}),\label{phimap}\end{equation}
see \cite[theorem 7.11]{gpor} and Proposition \ref{classinv}. 
Thus, without loss of generality we can and will consider that $F$ has the form (\ref{phimap}), 
hence, $w(s)\equiv1$. Then 
$t_L(s)=s$ up to additive constant. Thus, we have to show 
that 
\begin{equation} H_1(s,0)=\alpha s+\beta,\label{hsab}\end{equation}
\begin{equation} H_1(s,0)=s+\beta, \text{ if } H \text{ is symplectic.}\label{hsabs}\end{equation}
By definition, $H(s,y)=(H_1(s,y),H_2(s,y))$ conjugates $F$ to $\La:(t,z)\mapsto(t+\sqrt z,z)$. 
Hence, it sends the fixed point line $\{ y=0\}$ of the map $F$ to the fixed point line $\{ z=0\}$ of the map 
$\La$, thus, $H_2(s,0)\equiv0$. Writing conjugacy equation on the first components yields
$$H_1\circ F(s,y)=H_1(s+\sqrt y+O(y), y+o(y^{\frac32}))=H_1(s,0)+\frac{\partial H_1}{\partial s}(s,0)\sqrt y+O(y)$$
$$=\La_1\circ H(s,y)=H_1(s,y)+\sqrt{H_2(s,y)}=H_1(s,0)+\sqrt{\frac{\partial H_2}{\partial y}(s,0)}\sqrt y+O(y).$$
This yields 
\begin{equation}\frac{\partial H_1}{\partial s}(s,0)\equiv \sqrt{\frac{\partial H_2}{\partial y}(s,0)}>0.\label{dersq}\end{equation}
The Jacobian matrix of the map $H$ at points $(s,0)$ is equal to 
\begin{equation}
Jac(s,0)=\frac{\partial H_1}{\partial s}(s,0)\frac{\partial H_2}{\partial y}(s,0)=\left(\frac{\partial H_1}{\partial s}(s,0)\right)^3>0,
\label{symjac}\end{equation}
by (\ref{dersq}) and since $H_2(s,0)\equiv0$, which yields $\frac{\partial H_2}{\partial s}(s,0)=0$. 
This proves orientation-preserving property of the diffeomorphism $H$ and increasing of the function 
$H_1(s,0)$. 

Let now $H$ be symplectic, that is $Jac(s,0)\equiv1$. Then $\frac{\partial H_1}{\partial s}(s,0)\equiv1$, 
by (\ref{symjac}). This means that $H_1(s,0)=s+\beta$ for some $\beta\in\rr$. This proves (\ref{hso}).

Let now $H$ be not necessarily symplectic. Let us prove (\ref{hsab}). 
Suppose the contrary: there exist two points $s_0<s^*_0\in(a,b)$ such that 
$$\ell:=\frac{\partial H_1}{\partial s}(s_0,0)\neq\ell^*:=\frac{\partial H_1}{\partial s}(s^*_0,0).$$
Fix  small $\var,\delta>0$ such that 
$$s_0-\var, s_0^*+\var\in(a,b), \ \ [\ell-\delta,\ell+\delta]\cap[\ell^*-\delta,\ell^*+\delta]=\emptyset.$$ 
Fix a small $\eta>0$ and a $y_0\in(0,\frac\eta4)$, set $q_0=(s_0,y_0)$. 
Let $q_{-N_-},\dots, q_{-1},q_0,$ $q_1,\dots,q_{N_+}$, $q_j=(s_j,y_j)$, denote 
the $F$-orbit  
 of the point $q_0$ in the rectangle $[s_0-\var,s_0^*+\var]\times[0,\eta]$.  
  Here $N_\pm=N_\pm(y_0)$. 
It is known that the $s$-coordinates of its points form an asymptotic arithmetic progression $s_j=s(q_j)$, and 
their $y$-coordinates are asymptotically equivalent: 
\begin{equation}s_{j+1}-s_j\simeq\sqrt{y_0},  \ y_j\simeq y_0, \text{ as } y_0\to0,
\label{arprogr}\end{equation} uniformly in $j\in[-N_-(y_0),N_+(y_0)-1]$,
\begin{equation} s_{-N_-}<s_0, \ s_{N_+}>s_0^*,\label{predely}\end{equation}
whenever $y_0$ is small enough (dependently on $\var$). See \cite[lemma 7.13]{gpor}. The image of the above orbit under the map $H$ should be an orbit 
of the map $\La:(t,z)\mapsto(t+\sqrt z,z)$. The abscissas  of its points, $x_j:=H_1(q_j)$, form 
 an arithmetic progression: $x_{j+1}-x_j=\sqrt{z_0}$, $z_0=z(H(q_0))$. 
We claim that this yields a contradiction to the inequality 
$\ell\neq\ell^*$ and (\ref{arprogr}). Indeed, one has 
\begin{equation}x_1-x_0=H_1(q_1)-H_1(q_0)\simeq\ell(s_1-s_0)\simeq\ell\sqrt{y_0},\label{incr01}
\end{equation}
by (\ref{arprogr}) and the Lagrange Increment Theorem. 
On the other hand, take  a family of indices $k=k(y_0)$ such that $s_k=s_k(y_0)\to s_0^*$, as $y_0\to0$: 
it exists, since the asymptotic progression $s_j$ has steps uniformly decreasing to $0$, it starts 
on the left from $s_0$ and ends on the right from $s_0^*>s_0$, see (\ref{predely}). 
Repeating the above argument for $x_k$ and $x_{k+1}$ yields 
$$x_{k+1}-x_k\simeq\ell^*\sqrt{y_0}\neq
x_1-x_0\simeq\ell\sqrt{y_0},$$
whenever $y_0$ is small enough, since $\ell\neq\ell^*$. The contradiction thus obtained to the equality of the above differences 
proves that $\frac{\partial H_1}{\partial s}(s,0)\equiv const$. This proves (\ref{hsab}) and Proposition 
\ref{proform}. 
\end{proof}

\begin{proof} {\bf of 
 Proposition \ref{proform2}.} It repeats the proof of Proposition \ref{proform}, statement 2). 
 \end{proof}

\begin{proof} {\bf of Theorem \ref{thconj4}.} Let us fix an arbitrary point in each curve $\gamma_i$ and parametrize 
it  by natural length parameter so that the given point corresponds to zero parameter value. 
For every $i=1,2$ let 
 $\mct_{\gamma_i}$ denote the billiard map corresponding to the 
curve $\gamma_i$. It is a $C^\infty$-lifted strongly billiard-like map in the coordinates $(s,y)$ defined on a 
domain in $\rr\times\rr_+$ adjacent to an interval $J_{\gamma_i}=(a_i,b_i)\times\{0\}$. 
The corresponding function $w=w_i(s)$ is equal to $2\sqrt2\kappa_i^{-1}(s)$, where $\kappa_i$ is the curvature of 
the curve $\gamma_i$. 
There exists a domain $U_i\subset\rr_s\times(\rr_+)_y$  adjacent to $J_{\gamma_i}$ such that there exists a $C^\infty$-smooth symplectomorphism 
$H_i=(H_{1i},H_{2i})$ on $U_i\cup J_{\gamma_i}$ conjugating $\mct_{\gamma_i}$  to its normal form $\La:(t,z)\mapsto(t+\sqrt z,z)$, $H_i\circ \mct_{\gamma_i}\circ H_i^{-1}=\La$  
(Theorem \ref{addthm3}). The restriction to the $s$-axis of the first component $H_{1i}$ is given by the 
Lazutkin parameter:
\begin{equation}H_{1i}(s,0)=t_L(s):=\int_{0}^sw_i^{-\frac23}(u)du+const=\frac1{2}\int_{0}^s\kappa_i^{\frac23}(u)du+const,
\label{newstar}\end{equation}
by Proposition \ref{proform}. Therefore, the image $H_{1i}(J_{\gamma_i})$ is the interval $\wt J_i\times\{0\}$ 
equipped with the coordinate $t_L$, whose length is thus equal to $\frac1{2}\mcl(\gamma_i)$. 
The image  domain $H_i(U_i)\subset\rr\times\rr_+$ is adjacent to $\wt J_i$. Thus, (symplectic) $C^\infty$-conjugacy 
of the billiard maps near the boundary  is equivalent  to the existence of a (symplectic) 
$C^\infty$-diffeomorphism $\Phi$ commuting with $\La$, defined on a domain  $V_1\subset\rr_t\times(\rr_+)_z$ adjacent to 
$\wt J_1$ and sending it onto a domain $V_2\subset\rr_t\times(\rr_+)_z$ adjacent to $\wt J_2$ that extends 
as a $C^\infty$-diffeomorphism to $\wt J_1$, $\Phi(\wt J_1)=\wt J_2$. The latter diffeomorphism $\Phi$ exists 
in the class of symplectomorphisms, if and only if $\mcl(\gamma_1)=\mcl(\gamma_2)$ and one of the 
conditions i) or ii) of Theorem \ref{thconj3} holds. Indeed, if a symplectomorphism $\Phi$ commuting with 
$\La$ exists, then its restriction to $\wt J_1$ should be a translation (Proposition \ref{proform}). This implies 
that the lengths of the intervals $\wt J_1$, $\wt J_2$ are equal (i.e.,  the Lazutkin lengths of the curves $\gamma_1$, 
$\gamma_2$ are equal) and at least one of the conditions i) or ii) holds. Conversely, if 
$\mcl(\gamma_1)=\mcl(\gamma_2)$ and one of the conditions i) or ii) holds, then we can and will consider 
that $\wt J_1=\wt J_2$, applying a translation. Then the identity symplectomorphism $\Phi=Id$ 
has the required properties. This proves Theorem \ref{thconj4}. 
\end{proof}

\begin{proof} {\bf of Theorem \ref{thconj3}.} Let us repeat the above argument, where now the above diffeomorphism 
$\Phi:V_1\cup\wt J_1\to V_2\cup\wt J_2$ commuting with $\La$ is not necessarily symplectic. Let such a $\Phi$ 
exist. Then its restriction to $\wt J_1$ is an affine map in the first coordinate, $t\mapsto\alpha t+\beta$ 
(Proposition \ref{proform}), and $\wt J_2=\Phi(\wt J_1)$. 
This implies that $\mcl(\gamma_2)=\alpha\mcl(\gamma_1)$, and  one of the conditions i) or ii) holds. 
Conversely, let one of the conditions i) or ii) holds. Then without loss of generality we can and will consider that 
the interval $\wt J_2$ is obtained from $\wt J_1$ by a homothety $t\mapsto\alpha t$, applying a translation. 
The latter homothety 
extends to the linear map $\Phi:(t,z)\mapsto(\alpha t,\alpha^2z)$ commuting with $\La$ and thus, having the 
required properties. The statement of Theorem \ref{thconj3} on conjugacy in $(s,y)$-coordinates is proved. 
Together with Remark \ref{rsysf}, it implies that if one of the conditions i) or ii) holds, then the billiard 
maps are $C^\infty$-smoothly conjugated near the boundary in $(s,\phi)$-coordinates. Let us prove the converse: 
$C^\infty$-smooth conjugacy in $(s,\phi)$-coordinates implies that one of the conditions i) or ii) holds. 

Let $\wt H$ be a $C^\infty$-smooth diffeomorphism conjugating the billiard maps $\mct_{\gamma_1}$ and 
$\mct_{\gamma_2}$ near the boundary in $(s,\phi)$-coordinates, 
$\wt H\circ\mct_{\gamma_1}\circ\wt H^{-1}=\mct_{\gamma_2}$.
Let $H_i$ be symplectomorphisms conjugating  the billiard maps $\mct_{\gamma_i}$ in 
$(s,y)$-coordinates with the map 
$\La:(t,z)\mapsto(t+\sqrt z,z)$, see the above proof of Theorem \ref{thconj4}. The variable changes 
$$y\mapsto\phi=\arccos(1-y), \ \ z\mapsto\wt z:=\sqrt z$$ 
lift each diffeomorphism $H_i$ to a diffeomorphism $\wh H_i(s,\phi)=(\wh H_{1i}(s,\phi),\wh H_{2i}(s,\wt\phi))$ 
 conjugating the corresponding billiard map with 
the map $\wt\La:(t,\wt z)\mapsto(t+\wt z,\wt z)$. 
Then the restriction of its first component $\wh H_{1i}(s,0)=H_{1i}(s,0)$ to the 
$s$-axis coincides with the Lazutkin parameter $t_L$ of the curve $\gamma_i$ up to  post-composition with 
affine transformation,  by Proposition \ref{proform}. Set $\wh J_i:=H_{1i}(J_{\gamma_i})$. 
The diffeomorphism $\Phi:=\wh H_2\circ\wt H\circ\wh H_1^{-1}$ 
commutes with $\wt\La$, sends the interval $\wh J_1$ onto $\wh J_2$, and it sends a domain in 
$\rr\times\rr_+$ adjacent to $\wh J_1$ onto a domain in $\rr\times\rr_+$ adjacent to $\wh J_2$. Therefore, 
the  restriction of its first component to $\wh J_1$ is an affine map $t\mapsto\alpha t+\beta$, by 
Proposition \ref{proform2}. Thus, $\wh J_2$ is a rescaled image of the interval $\wh J_1$ up to translation. 
Recall that the lengths of the intervals $\wh J_i$ are equal to the Lazutkin lengths of the corresponding curves 
$\gamma_i$ divided by $2$, see (\ref{newstar}). This implies that the improper integrals defining the 
Lazutkin lengths of the curves $\gamma_i$ converge or diverge simultaneously, and one of the conditions 
i) or ii) holds. Theorem \ref{thconj3} is proved.
\end{proof}

\begin{proof} {\bf of Lemma \ref{lazconv}.} Consider a strictly convex $C^2$-smooth planar 
curve $\gamma$ going to infinity that 
has an asymptotic tangent line at infinity. Without loss of generality we can and will consider that 
the latter tangent line is the horizontal $x$-axis in $\rr^2_{x,y}$,  $\gamma$ lies above it, and 
$\gamma$ is the graph of a $C^2$-smooth function $f$ defined on $[1,+\infty)$, $\gamma=\{ y=f(x)\}$,
\begin{equation}f(x),f'(x)\to0, \text{ as } x\to+\infty, \ \ f,f''>0, \ f'<0.\label{asbove}\end{equation}
 One can achieve this by  applying 
rotation sending the asymptotic line to the $x$-axis, restricting ourselves to a subarc of the curve 
$\gamma$ with the same asymptotic line and applying symmetry with respect to the $x$-axis, if necessary. 
The improper integral defining the Lazutkin length of the curve $\gamma$ takes the form 
\begin{equation}\int_1^{+\infty}\kappa^{\frac23}(s)ds, \ \ \kappa(s(x))=\frac{f''(x)}{(1+(f'(x))^2)^{\frac32}}, \ 
ds=\sqrt{1+(f'(x))^2}dx.\label{implaz}\end{equation}
Its convergence is equivalent to the convergence of the integral
\begin{equation}\int_1^{+\infty}(f''(x))^{\frac23}dx,\label{imprint}\end{equation}
since $f'(x)\to0$, thus, $1+(f'(x))^2\to1$, as $x\to+\infty$. 

{\bf Claim.} {\it For every $C^2$-smooth function $f$ as in (\ref{asbove}) the improper integral  (\ref{imprint}) converges.}

\begin{proof} The integral (\ref{imprint}) is estimated from above by H\"older inequality:
$$\int_1^{+\infty}(f''(x))^{\frac23}dx=\int_1^{+\infty}(xf''(x))^{\frac23}x^{-\frac23}dx$$
\begin{equation}
\leq\left(\int_1^{+\infty}xf''(x)dx\right)^{\frac23}\left(\int_1^{+\infty}\frac{dx}{x^2}\right)^{\frac13}.\label{holder}
\end{equation}
Therefore, it remains to prove that the integral $\int_1^{+\infty}xf''(x)dx$ converges. Integrating by parts yields 
\begin{equation}\int_1^{+\infty}xf''(x)dx=xf'(x)|_1^{+\infty}-\int_1^{+\infty}f'(x)dx=xf'(x)|_1^{+\infty}-f(1).
\label{intbyp}\end{equation}
Suppose the contrary: the integral in the left-hand side diverges. Then it is equal to $+\infty$, since $f''(x)>0$. 
Therefore,  $xf'(x)\to+\infty$, as $x\to+\infty$. Hence, $f'(x)>\frac1x$, whenever $x$ is greater than some constant 
$N>1$. Integrating the latter inequality along the interval $[N,+\infty)$ yields $f(N)=+\infty$. 
The contradiction thus obtained proves convergence of the integral in the left-hand side in (\ref{intbyp}), 
and hence, of the integral (\ref{imprint}). The claim is proved.
\end{proof}

The claim together with (\ref{implaz}) imply convergence of the improper integral defining the Lazutkin length. 
This proves Lemma \ref{lazconv}.
 \end{proof}
 
 \begin{proof} {\bf of Theorem \ref{thconj1}.} The Lazutkin lengths of both curves $\gamma_1$ and $\gamma_2$ 
 are finite, since they have asymptotic tangent lines at infinity in both directions and by Lemma \ref{lazconv}. 
 This together with Theorem \ref{thconj3} implies $C^\infty$-smooth 
 conjugacy of the corresponding billiard  maps near the boundary 
 and up to the boundary. Theorem \ref{thconj1} is proved. 
 \end{proof}
 
 The proof of Theorem \ref{thconj2} is analogous to the above proof of Theorem \ref{thconj1}.

\section{Acknowledgements}

I am grateful to Marie-Claude Arnaud, Sergei Tabachnikov, Alexander Plakhov, Ivan Beschastnyi, Stefano Baranzini, 
Vadim Kaloshin, Comlan Koudjinan for helpful discussions. I am grateful to the referee for careful reading the paper and for very helpful remarks.

\end{document}